\documentclass[reqno,a4paper,12pt]{amsart}
\usepackage{hyperref} 
\usepackage{amscd}
\usepackage[sc]{mathpazo}
\numberwithin{equation}{section}

\newlength{\myfboxsep}
\newlength{\mywidth}
\usepackage{amsmath}
\usepackage{amssymb,bbm,mathrsfs}
\usepackage{graphicx}
\usepackage[english]{babel}
\usepackage[usenames]{color}
\usepackage{ifpdf}
\usepackage{rotating}
\usepackage[table]{xcolor}

\newtheorem{theorem}{Theorem}[section]
\newtheorem{lemma}[theorem]{Lemma}
\newtheorem{definition}[theorem]{Definition}
\newtheorem{proposition}[theorem]{Proposition}
\newtheorem{corollary}[theorem]{Corollary}
\newtheorem*{corollary*}{Corollary}

\theoremstyle{remark}
\newtheorem{remark}{Remark}[section]
\numberwithin{equation}{section}
\newcommand{\R}{\mathbb{R}}
\newcommand{\C}{\mathbb{C}}

\newcommand{\Z}{\mathbb{Z}}
\newcommand{\N}{\mathbb{N}}
\newcommand{\eps}{\epsilon}
\newcommand{\la}{\lambda}

\newcommand{\vth}{\vartheta}
\newcommand{\Pb}[1]{\mathbb{P}\left[#1\right]}
\newcommand{\E}[1]{\mathbb{E}\left[#1\right]}
\newcommand{\Pbn}[1]{\mathbb{P}_n\left[#1\right]}
\newcommand{\En}[1]{\mathbb{E}_n\left[#1\right]}
\newcommand{\Pbt}[1]{\mathbb{P}_t\left[#1\right]}
\newcommand{\Et}[1]{\mathbb{E}_t\left[#1\right]}
\newcommand{\one}[1]{\mathbbm{1}_{\left\{#1\right\}}} 

\newcommand{\slan}{\sum_{\la \vdash n}\frac{1}{z_\la}}

\renewcommand{\i}{\imath}
 \newcommand{\set}[1]{\left\{#1\right\}}
 \newcommand{\setn}{\{1,\ldots,n\}}
 \renewcommand{\Re}[1]{\mathop{\mathfrak{Re}}\left(#1\right)} 
\newcommand{\floor}[1]{\lfloor #1 \rfloor}

\newcommand{\eq}{\begin{equation}}
\newcommand{\eeq}{\end{equation}}

\renewcommand{\exp}[1]{ e^{#1}}
\newcommand{\abs}[1]{\left| #1 \right|}
\renewcommand{\O}[1]{O\left(#1 \right)} 
\renewcommand{\o}[1]{o\left(#1 \right)} 
\newcommand{\vr}{\varphi}

\usepackage{subfigure}

\begin{document}
\title[Limit shape of random permutations]{The limit shape of random permutations with polynomially growing cycle weights}

\author[A. Cipriani]{Alessandra Cipriani}
\address{Weierstra{\ss}-Institut\\
Mohrenstra{\ss}e 39\\
10117 Berlin\\
Germany}
\email{Alessandra.Cipriani@wias-berlin.de}

\author[D. Zeindler]{Dirk Zeindler}
\address{
Sonderforschungsbereich 701\\
Fakult\"at f\"ur Mathematik\\
Universit\"at Bielefeld\\
Postfach 10 01 31\\
33501 Bielefeld\\
Germany} 
\email{zeindler@math.uni-bielefeld.de}

\date{\today}

\begin{abstract}
In this work we are considering the behaviour of the limit shape of Young diagrams associated to random permutations on the set $\setn$ under a particular class of multiplicative measures
with polynomial growing cycle weights. 
Our method is based on generating functions and complex analysis (saddle point method).
We show that fluctuations near a point behave like a normal random variable and that the joint fluctuations at different points of the limiting shape have an unexpected dependence structure.
We will also compare our approach with the so-called randomization of the cycle counts of permutations and we will study the convergence of the limit shape to a continuous stochastic process.
\end{abstract}
\maketitle

\section{Introduction and main results}
The aim of this paper is to study the limit shape of a random permutation under the generalised Ewens measure 
with polynomially growing cycle weights and the fluctuations at each point of the limit shape. 
The study of such objects has a long history, which started with the papers of Temperley \cite{Temp52} and Vershik \cite{Ve96}. Later on Young diagrams 
have been approached under a different direction, as in the independent works of \cite{KeVe77} and \cite{LoSh77}, which first derived the limit 
shape when the underpinned measure on partitions is the so-called \emph{Plancherel measure}. We will not handle this approach here, even though it presents remarkable connections 
with random matrix theory and random polymers, among others (see for example \cite{DeiftInt}).

We first specify what we define as the limit shape of a permutation.
We denote by $\mathfrak S_n$ the set of permutations on $n$ elements and write each permutation  $\sigma\in \mathfrak  S_n$ as $\sigma = \sigma_1\cdots \sigma_\ell$ 
with $\sigma_j$ disjoint cycles of length $\la_j$.
Disjoint cycles commute and we thus can assume $\la_1\geq\la_2\geq\cdots \geq\la_\ell$.
This assigns to each permutation $\sigma\in\mathfrak S_n$ in a unique way a partition of $n$ and this partition $\la=(\la_1,\la_2,\dots,\la_\ell)$ is called the \emph{cycle type} of $\sigma$. We will indicate that 
$\la$ is such a partition with the notation $\la \vdash n$. We define the size $|\la|:=\sum_{i}\la_i$ (so obviously if $\la \vdash n$ then $\abs{\la}=n$). $\la$ features a nice geometric visualisation by its Young diagram $\Upsilon_\la$. 
This is a left- and bottom-justified diagram of $\ell$ rows with the $j-$th row consisting of $\la_j$ squares,  see Figure~\ref{fig:YoungDiagram}.
\begin{figure}[ht!]
 \centering

\subfigure[\,The Young diagram]
{\hspace{-1pc}
 \includegraphics[height=.2\textheight]{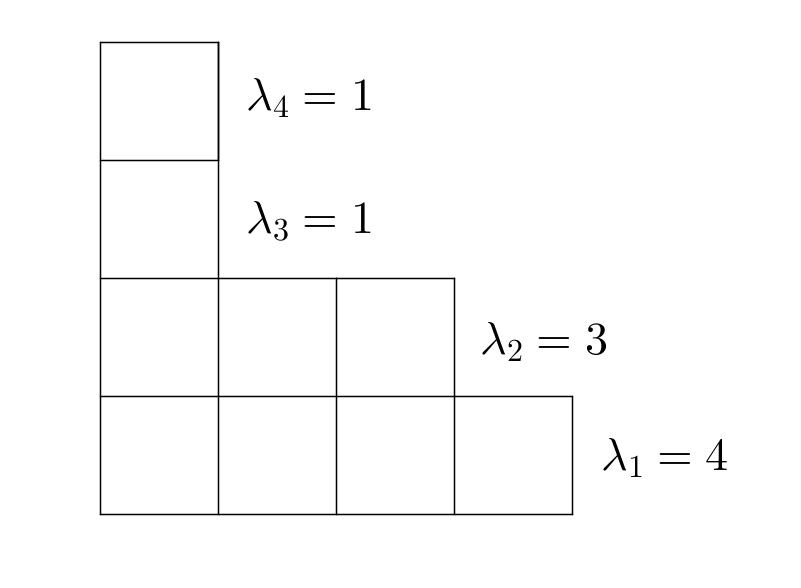}           
\label{fig:YoungDiagram}
 }
\subfigure[\,The shape function $w_n(\cdot)$]
{\hspace{2pc}
 \includegraphics[height=.25\textheight]{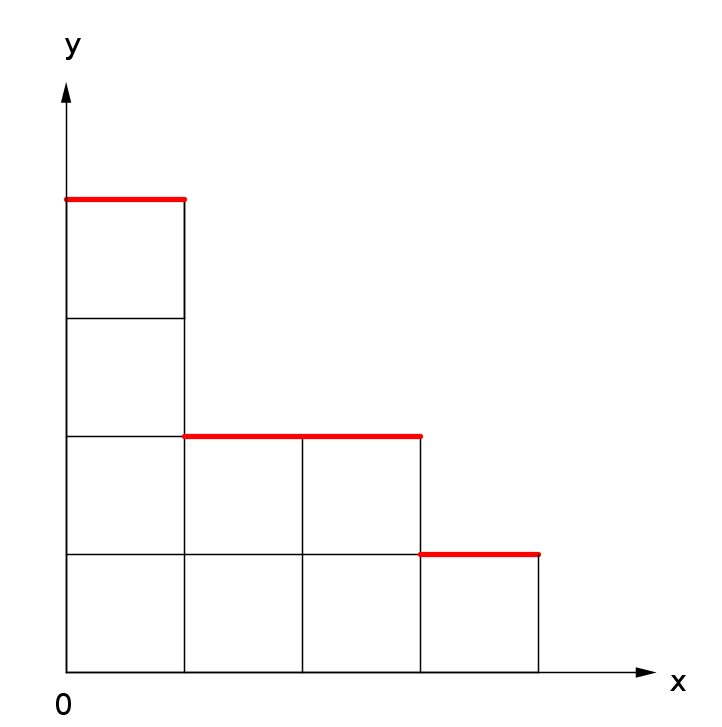}           
\label{fig:limit_shape}
 }
%
%
\caption{Illustration of the Young diagram and the shape of \newline $\sigma=(3578)(129)(4)(6)\in  \mathfrak S_9$}
\end{figure} 
It is clear that the area of $\Upsilon_\la$ is $n$ if $\la\vdash n$. 
 After introducing a coordinate system as in Figure~\ref{fig:limit_shape}, we see that
the upper boundary of a Young diagram $\Upsilon_\la$ is a piecewise constant and right continuous function $w_n:\R^+\to\N$ with 
\begin{align}
\label{eq:def_w_n_part}
w_n(x)
:=
\sum_{j=1}^n \one{\la_j\geq x}
\end{align}
The cycle type of a permutation becomes a random partition if we endow the space 
$\mathfrak S_n$ with a probability measure $\mathbb{P}_n$. What we are then interested in studying is the now random shape $w_n(\cdot)$ as $n\to+\infty$, and more specifically to determine its limit shape.
The limit shape with respect to a sequence of probability measures $\mathbb{P}_n$ on $\mathfrak S_n$ 
(and sequences of positive real numbers $A_n$ and $B_n$ with $A_n \cdot B_n =n$) is understood as a function $w_\infty:\R^+ \to \R^+$ 
such that for each $\eps,\delta>0$
\begin{align}
\label{eq:def_limit_shape}
\lim_{n\to+\infty} \mathbb{P}_n \left[\set{\sigma\in\mathfrak{S}_n:\, \sup_{x\geq \delta} |A_n^{-1} w_n (B_n x)- w_\infty(x)| \leq \eps }  \right]
= 1.
\end{align}
The assumption $A_n \cdot B_n =n$ ensures that the area under the rescaled Young diagram is $1$.
One of the most frequent choices is $A_n=B_n=n^{1/2}$, but we will see that it's useful to adjust the choice of $A_n$ and $B_n$ to the measures $\mathbb{P}_n$. 
Equation~\eqref{eq:def_limit_shape} can be viewed as a law of large numbers for the process $w_n(\cdot)$. The next natural question is then whether fluctuations satisfy a central limit theorem, namely whether 
$$
A_n w_n(B_n x)-w_\infty(x)
$$
converges (after centering and normalization) in distribution to a Gaussian process on the space of c\`adl\`ag functions, for example. Of course the role of the probability distribution with which we equip the set of partitions will be crucial to this end.\\

In this paper, we work with the following measure on $\mathfrak S_n$:
\begin{align}
\label{eq:def_weighted_pb_la}
 \Pbn{\sigma}=\frac{1}{h_n n!}\prod_{j= 1}^\ell \vth_{\la_j}.
\end{align}
where $(\la_1,\dots,\la_\ell)$ is the cycle type of $\sigma$, 
$(\vth_m)_{m\geq 1}$ is a sequence of non-negative weights and $h_n$ is a normalization constant ($h_0$ is defined to be $1$). From time to time we will also use $\vth_0:=0$ introduced as convention.

This measure has recently appeared in mathematical physics for a model of the quantum gas in statistical
mechanics and has a possible connection with the Bose-Einstein condensation (see e.g. \cite{BeUeVe11} and \cite{ErUe11}). 
Classical cases of this measure are the uniform measure ($\vth_m \equiv 1$) and
the Ewens measure ($\vth_m \equiv \vth$). The uniform measure is well studied and has a long history
(see e.g. the first chapter of \cite{ABT02} for a detailed account with references).
 The Ewens measure originally appeared in population genetics, 
see \cite{Ew72}, but has also various applications through its connection with Kingman's coalescent process, see \cite{Ho87}.
The measure in \eqref{eq:def_weighted_pb_la} also has some similarities to multiplicative measure for partitions, see for instance \cite{Bo11a}. 
It is clear that we have to make some assumptions on the sequence $(\vth_m)_{m\geq 1}$ to be able to study the behaviour as $n\to+\infty$.
We use in this paper the weights $\vth_m$ 
\begin{equation}\label{eq:weights}
 \vth_m=(\log m)^j \frac{m^\alpha}{\Gamma(\alpha+1)}+\O{m^\beta},\quad j\in \N
 \end{equation}

with some $\alpha>0$ and $0\leq \beta < \alpha/2$. 
We would like to point out that the requirement $0\leq \beta < \alpha/2$ 
and the normalisation constant $\Gamma(\alpha+1)$
are not essential and it only simplifies the notation and the computations. 
In fact we can study without further mathematical problems the case $\vth_m \sim \mathtt{const.}(\log m)^j  m^\alpha$, see Remark~\ref{rem:general_weights}. We get


\begin{theorem}
\label{thm:limit_shape_saddle}\mbox{}
We define 
\begin{align}\label{eq:def_n*}
 n^*:= 
 (1+\alpha)^{-j}\left(\frac{n}{\left(\log n\right)^j}\right)^{\frac1{\alpha+1}}
 \ \text{ and } \
 \overline n= \frac{n}{n^*}= (1+\alpha)^{j} n\left(\frac{n}{\left(\log n\right)^j}\right)^{-\frac1{\alpha+1}}.
\end{align}
We then have
\begin{enumerate}
 \item The limit shape exists for the process $w_n(x)$ as $n\to\infty$ with respect to $\mathbb{P}_n$ and the weights in \eqref{eq:weights} with the scaling $A_n = \overline n$ and $B_n=n^*$. 
 The limit shape is given by
 $$
 w_\infty^{\mathbf s}(x):=\frac{\Gamma(\alpha,x)}{\Gamma(\alpha+1)},
 $$
 where $\Gamma(\alpha,x)$ denotes the upper incomplete Gamma function.  
 \item The fluctuations at a point $x$ of the limit shape behave like
 \begin{eqnarray*}&&
\widetilde w_n^{\mathbf s}(x)
:=
\frac{w_n(xn^*)-\overline n\left( w_\infty^{\mathbf s}(x)+z_n^{\mathbf s}(x)\right)}{(\overline n)^{1/2}}
\stackrel{\mathcal L}{\longrightarrow}
\mathcal N\left(0, \sigma_\infty^2(x)\right)
\end{eqnarray*}
with
 $$
 \sigma_\infty^2(x):= \frac{\Gamma(\alpha,x)}{\Gamma(\alpha+1)}-\frac{\Gamma(\alpha+1,\,x)^2}{2 \Gamma(\alpha+1)\Gamma(\alpha+2)}
 $$
 and $z_n^{\mathbf s}(x)=\o{1}$.
 \end{enumerate}
\end{theorem}
\begin{remark}
 The expectation of $w_\infty^{\mathbf s}$ can be expanded asymptotically also to terms of lower order with the same argument. 
 This will actually be important in the proof of Thm.~\ref{thm:func_CLT_pt}. For the time being however we want to stress only the leading coefficient of the expansion.
\end{remark}
%

Theorem~\ref{thm:limit_shape_saddle} was already obtained in the special case $j=0$, i.e. $\vth_m \sim m^\alpha$,  by Erlihson and Granovsky in \cite{ErGr08} in the context of Gibbs distributions on integer partitions and 
as we were writing the present paper, we were made aware of their work.
To be precise, one can push forward the measure $\mathbb{P}_n$ to a measure $\widetilde{\mathbb{P}}_n$ on the set of partitions of $n$ with
 \begin{align}
  \widetilde{\mathbb{P}}_n[\la]
  =
  \frac{1}{\widetilde{h}_n n!}\prod_{k= 1}^n \frac{1}{C_k!}\left(\frac{\vth_{m}}{m}\right)^{C_k},
\end{align}
where $\la$ is a partition of $n$ and $C_k$ is the number of parts of length $k$ (see Section~\ref{sec:cycle_counts}). 
These Gibbs distributions have been treated extensively in the literature (\cite{BeUeVe11}, \cite{ErUe11} for example).
One thus can work with $\mathbb{P}_n$ or with $\widetilde{\mathbb{P}}_n$. We prefer here to use $\mathbb{P}_n$.


The argumentation of Erlihson and Granovsky in \cite{ErGr08}
is stochastic and is based on randomisation: this technique has been successfully introduced by \cite{Fri93} and used also in particular by \cite{Bo11a} as a tool to investigate combinatorial structures, 
and later applied in many contexts. However, the approach in this paper is slightly different and bases on complex analysis and uses the saddle-point method 
as described in Section~\ref{sec:estmgf}. This method was used in \cite{ErUe11} and \cite{MaNiZe11} and 
an introduction can be found for instance in \cite[Section~VIII]{FlSe09}.
Our Ansatz enables us to reprove Theorem~\ref{thm:limit_shape_saddle}, but with two big advantages.
First, our computations are much simpler than the one in \cite{ErGr08}. Second, we get almost for free large deviations estimates. More precisely
 \begin{proposition}\label{prop:forall}
We have for all $a=\O{\overline n}$ and $\delta=\O{(\overline n)^{-1/2}}$
\begin{align*}
&\Pb{ \abs{\widetilde{w}^{\mathbf s}_n(x)-a}<\epsilon}
=
\left(1-\eps^{-2}(1+\delta)\right)\mathrm{exp}\left(-a^2/2+\O{\delta+\epsilon a}\right).
\end{align*}
The error terms are absolute.
\end{proposition}
We prove this by studying the cumulant generating function
\begin{align}
\widetilde{\Lambda}(s)
:=
 \log \mathbb{E}_n \left[\mathrm{exp}\bigl(s \widetilde{w}^{\mathbf s}_n(x)\bigr)\right].
\end{align}
This is an important difference to \cite{ErGr08}. Erlihson and Granovsky directly consider the distribution of $\widetilde w_n^{\mathbf s}(x)$
and studying $\widetilde{\Lambda}(s)$ with their method is computationally harder.
In fact, we can compute the behaviour of all cumulants. 
\begin{theorem}\label{thm:cumulants}
Let $s^* = s(\overline n)^{-\frac{1}{2}}$ and 
\begin{equation}
 \label{eq:cum_generating_func}\Lambda(s)
:=
\En{\mathrm{exp}\bigl(-s^\ast w_n(xn^\ast)\bigr)}
=
\sum_{m\geq 1}q_m \frac{s^m}{m!}.
\end{equation}
We then have for $m\geq 2$
\begin{align}
 q_m = \kappa_m(1+o(1))
\end{align}
with
\begin{eqnarray}
 \kappa_m
&=&
(\overline n)^{1-\frac{m}{2}}[s^m]
\left[
\left(1-s\frac{\Gamma(\alpha+1,\,x)}{\Gamma(\alpha+2)}\right)^{-\alpha}\right.\nonumber\\
&&+
\left. (\exp{-s}-1)\sum_{k=0}^\infty  \frac{s^k}{k!}\frac{\Gamma(\alpha+k,x)}{\Gamma(\alpha+1)} \left(-\frac{\Gamma(\alpha+1,\,x)}{\Gamma(\alpha+2)}\right)^k
\right].
\end{eqnarray}
\end{theorem}

We give the proofs of Theorem~\ref{thm:limit_shape_saddle}, Theorem~\ref{thm:cumulants} and Proposition~\ref{prop:forall} in Section~\ref{sec:calculation}.
Furthermore, we introduce in Section~\ref{sec:randomization} the so called grand canonical ensemble $(\Omega_t, \mathbb{P}_t)$ with $\Omega_t = \stackrel{.}{\cup}_{n\geq 1}\mathfrak S_n$ 
and $\mathbb{P}_t$ a measure such that $\mathbb{P}_t[\,\cdot\,|\mathfrak S_n]= \mathbb{P}_n[\,\cdot\,]$ (see \eqref{eq:def_Pt}). 
It is widely expected that the behaviour on grand canonical ensembles agrees with the behaviour on the canonical ensembles, but we will see here that this is only the case for macroscopic properties.
More precisely, we will see in Theorem~\ref{eq:CLT_Pt} that $w_n(x)$ has a limit shape for the grand canonical ensemble $\Omega_t$ and this agrees with the one for the canonical ensemble in Theorem~\ref{thm:limit_shape_saddle}.
However, we will see also in Theorem~\ref{eq:CLT_Pt} that the fluctuations at the points of the limit shape follow a different central limit theorem than in Theorem~\ref{thm:limit_shape_saddle}.
Notice that we will not deduce Theorem~\ref{thm:limit_shape_saddle} (nor any other of our results) from the grand canonical ensemble $(\Omega_t, \mathbb{P}_t)$.

\section{Preliminaries}

We introduce in this section the notation of the cycle counts and the notation of generating functions.

\subsection{Cycle counts}
\label{sec:cycle_counts}

The notation $\la = (\la_1,\la_2,\dots,\la_\ell)$ is very useful for the illustration 
of $\la$ via its Young diagram, but in the computations it is better to work with
the \emph{cycle counts} $C_k$. These are defined as
\begin{align}
  C_k(\sigma)=C_k  := \#\set{j\geq1; \la_j = k}
\end{align}
for $k\geq1$ and $\la= (\la_1,\la_2,\dots,\la_\ell)$ the cycle type of $\sigma\in \mathfrak{S}_n$. Conventionally $C_0:=0$.
We obviously have for $k\geq 1$
\begin{equation}\label{eq:sumC}
C_k\ge0 \ \text{ and } \ \sum_{k=1}^n k C_k=n.
\end{equation}
It is also clear that the cycle type of permutation (or a partition) is uniquely determined by the vector $(C_1,C_2,\dots)$.
The function $w_n(x)$ and the measure $\Pbn{\,\cdot\,}$ in \eqref{eq:def_w_n_part} and \eqref{eq:def_weighted_pb_la} can now be written as
\begin{align}
\label{eq:def_weighted_pb_ck}
w_n(x)
=
\sum_{k\geq x}C_k
\ \text{ and } \
\Pbn{\sigma}
&=
\frac{1}{h_n n!}\prod_{k= 1}^n \vth_k^{C_k}.
\end{align}
%
%
Our aim is to study the behaviour of $w_n(x)$ as $n\to\infty$. It is thus natural 
to consider the asymptotic behaviour of $C_k$ with respect to the measure $\Pbn{\,\cdot\,}$. 
\begin{lemma}[\cite{ErUe11}, Corollary~2.3]\label{lemma:BetzUel}
Under the condition $h_{n-1}/h_n \to 1$ the random variables $C_k$ converge for each $k\in\N$ in distribution to a
Poisson distributed random variable $Y_k$ with $\mathbb E\left[Y_k\right]=\frac{\vth_k}{k}$. More generally for all $b \in \N$ the following limit in distribution holds:
$$
\lim_{n\to+\infty}\left(C_1,\,C_2\,\ldots,\,C_b\right)= \left(Y_1,\,Y_2\,\ldots,\,Y_b\right)
$$
with $Y_k$ independent Poisson random variables with mean $\mathbb E\left[Y_k\right]=\frac{\vth_k}{k}$.
\end{lemma}

One might expect at this point that 
$w_n(x)$ is close to $\sum_{k=x}^n Y_k$.
Unfortunately  we will see in Section~\ref{sec:estmgf} that the asymptotic behaviour of $w_n(x)$ is more complicated.

\subsection{Generating functions}
\label{sec:gfs}

The (ordinary) generating function of a sequence $(g_k)_{k\geq 0}$ of complex numbers
is defined as the formal power series
\begin{align}\label{eq:G}
g(z): = \sum_{j=0}^\infty g_k z^k.
\end{align}
As usual, we define the
\emph{extraction symbol} $[z^k]\, g(z):= g_k\label{def:ext}$, that is, as the
coefficient of $z^k$ in the power series expansion \eqref{eq:G}
of~$g(z)$.

A generating function that plays an important role in this paper is 
\begin{align}
\label{eq:def_g_theta}
 g_\Theta(t)
:=
\sum_{m\geq 1}\frac{\vth_m}{m}t^m.
\end{align}
As mentioned in the introduction, we will use $\vth_m=\frac{m^\alpha(\log m)^j}{\Gamma(\alpha+1)}+\O{m^\beta}$, $j\in \N$.
We stress that generating functions of the type $(1-t)^{-\alpha}$ fall also in this category, 
and for them we will recover the limiting shape as previously done in \cite{ErGr08}. We will see in particular this case in Section~\ref{sec:estmgf}.\\
The reason why generating functions are useful is that it is often possible to write down
a generating function without knowing $g_n$ explicitly. In this case one can 
try to use tools from analysis to extract information about $g_n$, for large $ n$, from the generating function. 
It should be noted that there are several variants in the definition of generating functions.
However, we will use only the ordinary generating function and thus 
call it `just' generating function without risk of confusion. 

The following well-known identity is a special case of the general \emph{P\'olya's
Enumeration Theorem} \cite[p.\,17]{Po37} and is the main tool in this paper to
obtain generating functions. 
\begin{lemma}\label{lemma:Polya}
Let $(a_m)_{m\in \N}$ be a sequence of complex numbers. We then
have as formal power series in $t$
\begin{eqnarray*}
 &&
\sum_{n\in \N}\frac{t^n}{n!}\sum_{\sigma\in \mathfrak S_n}\prod_{j=1}^n a_j^{C_j}
=
\sum_{n\in \N} t^n \slan  \prod_{k=1}^\infty a_k^{C_k}
=
\mathrm{exp}\left(\sum_{m\geq 1} \frac{a_m}{m}t^m\right)
\end{eqnarray*}
where $z_\lambda:=\prod_{k=1}^n k^{C_k} C_k!$. If one series converges absolutely, so do the others.
\end{lemma}

We omit the proof of this lemma, but details can
be found for instance in \cite[p.~5]{Mac95}.

%
%
%

\subsection{Approximation of sums}

We require for our argumentation the asymptotic behaviour of the generating function 
$g_\Theta(t)$ as $t$ tends to the radius of convergence, which is $1$ in our case. 

\begin{lemma}\label{lemma:polylog_asymp}
Let $(v_n)_{n\in \N}$ a sequence of positive numbers with $v_n\downarrow 0$ as $n\to+\infty$. We have for all $\delta\in\R\setminus\set{-1,\,-2,\,-3,\dots}$
\begin{align}
\sum_{k=1}^\infty k^\delta \exp{-k v_n}
=
\Gamma(\delta+1) v_n^{-\delta-1} + \zeta(-\delta) +O(v_n).
\end{align}
$\zeta(\cdot)$ indicates the Riemann Zeta function. 
 Furthermore, we have for $j\in\N$
\begin{align}
\sum_{k=1}^\infty (\log k)^j k^\delta \exp{-k v_n}
=
 v_n^{-\delta-1} \left(\frac{\partial}{\partial \delta}-\log v_n\right)^j\Gamma(\delta+1)+ O(1).
\end{align}
 We indicate $\left(\frac{\partial}{\partial \delta}-\log v_n\right)^j f(\delta)=\frac{\partial^j}{\partial \delta^j}f(\delta)-
 j\log v_n\frac{\partial^{j-1}}{\partial \delta^{j-1}}f(\delta)+\left({j \atop 2}\right)(\log v_n)^2\frac{\partial^{j-2}}{\partial \delta^{j-2}}f(\delta)+\ldots
 +\left(-\log v_n\right)^j f(\delta).$

\end{lemma}

This lemma can be proven with Euler Maclaurin summation formula or with the Mellin transformation.
The computations with Euler Maclaurin summation are straightforward and the details of the proof 
with the Mellin transformation can be found for instance in \cite[Chapter~VI.8]{FlSe09}.
We thus omit it.

We require also the behaviour of partial sum $\sum_{k=x}^\infty \frac{\theta_m}{m} t^m$ as $x\to\infty$ and as $t\to1 $. We have
\begin{lemma}[Approximation of sums]\label{lemma:approxsums}
Let $j\in\N$ and $v_n,z_n$ be given with $z_n\to+\infty$ and $z_n v_n = x(1+\widetilde{q}_n)$ for $x>0$ and $\widetilde{q}_n\to 0$.
We then have for all $\delta\in\R$ and all $\ell\in\N$ 
\begin{align*}
\sum_{k=\lfloor z_n \rfloor}^\infty (\log k)^j k^\delta \exp{-k v_n}
=&
\left(\frac{z_n}{x}\right)^{\delta+1}  \left(\sum_{k=0}^{\ell}  \left(\frac{\partial}{\partial \delta} +\log \frac{z_n}{x}\right)^j \frac{\Gamma(\delta+k+1,x)}{k!} \left(-\widetilde{q}_n\right)^{k}  \right)\\
&+ \O{\left(\log \left(\frac{z_n}{x}\right)\right)^j \widetilde{q}_n^{\ell+1} + \left(\frac{z_n}{x}\right)^{\delta} \left(\log\left(\frac{z_n}{x}\right) \right)^j}
\end{align*}

with $\Gamma(a,x):=\int_x^{+\infty} s^{a-1} \exp{-s}\mathrm d s$ the incomplete Gamma function.
\end{lemma}

\begin{remark}\label{rem:order_E}
 One can obtain more error terms by using the Euler Maclaurin summation formula with more derivatives. We have given in Appendix~\ref{sec:eulermac} a formulation of the Euler Maclaurin summation formula 
 with non-integer boundaries, which is more suitable for this computation than the usual one. Our primary interest is in the leading coefficient, hence we state the result only up to order $z_n^\delta$.
 However, the lower order terms can not be completely ignored. In particular they play an important role for the expectation of $\En{w_\infty^{\mathbf s}(x)}$ in Theorem~\ref{thm:limit_shape_saddle} 
 since there are, beside the leading term $\overline n \, w_\infty^{\mathbf s}(x)$, also other terms in the asymptotic expansion which are not $\o{(\overline n)^{1/2}}$. 
\end{remark}

%

%
%
%
%
%
%
%
%
\begin{proof}
$B_1(x):=x-\frac{1}{2}$ stands in the proof for the first Bernoulli polynomial. The proof of this lemma is based on the Euler Maclaurin summation formula, see \cite{Ap99} or \cite[Theorem~3.1]{Ap84}. 
We use here the following version: let $f:\R^+\to\R$ have a continuous derivative and suppose that $f$ and $f'$ are integrable. 
Then

\begin{align}
\label{eq:euler_mac_1_deri}
 \sum_{k\geq \floor{c}} f(k)
 =&
\int_{c}^{+\infty} f(s)\, \mathrm d s - B_1(c-\floor{c}) f(c)
+ 
\int_{c}^{+\infty}B_1(s-\lfloor s\rfloor) f'(s)\, \mathrm d s.
\end{align}
We substitute $f(s):=(\log s)^j s^\delta \exp{-s v_n}$, $c:=z_n$ and notice that $f$ and all its derivatives tend to zero exponentially fast as $s\to +\infty$. 
We begin with the first integral.
%
%
%
%
%
Now by the change of variables $s:=\frac{z_n}{x}y$
\begin{eqnarray}
&&\int_{z_n}^{+\infty}(\log s)^j s^\delta \exp{-v_n s} \mathrm d s
=
\left(\frac{z_n}{x} \right)^{\delta+1}\int_{x}^{+\infty} \left(\log y + \log \frac{z_n}{x} \right)^j  y^\delta \exp{-y}\exp{-\widetilde{q}_n y}\mathrm d y=\nonumber\\
&&=
\left(\frac{z_n}{x} \right)^{\delta+1} \left(\frac{\partial}{\partial \delta} + \log \frac{z_n}{x} \right)^j\int_{x}^{+\infty}  y^\delta \exp{-y}\exp{-\widetilde{q}_n y}\mathrm d y=\nonumber\\
&&=
\left(\frac{z_n}{x} \right)^{\delta+1} \left(\left(\frac{\partial}{\partial \delta}  + \log \frac{z_n}{x} \right)^j 
\sum_{k=0}^{\ell} \frac{\Gamma(\delta+k+1,x)}{k!} \left(-\widetilde{q}_n\right)^{k} + \O{\widetilde{q}_n^{m+1}}\right)\label{eq:firstpart}
\end{eqnarray}
where we have swapped integral and series expansion of the exponential by Fubini's theorem. 
This gives the behaviour of the leading term in \eqref{eq:euler_mac_1_deri} with $f(s):=(\log s)^j s^\delta \exp{-s v_n}$.
The remaining terms can be estimated with a similar computations and using that $B_1(s-\lfloor s\rfloor)$ is bounded.

\end{proof}

\section{Randomization}\label{sec:randomization}

We introduce in this section a probability measure $\Pbt{\,\cdot\,}$ on $\stackrel{.}{\cup}_{n\geq 1}\mathfrak S_n$, where $\stackrel{.}{\cup}$ denotes the disjoint union, dependent on
a parameter $t>0$ with $\Pbt{\,\cdot\,| \mathfrak S_n} = \Pbn{\,\cdot\,}$ and consider the asymptotic behaviour of $w_n(x)$ 
with respect to $\Pbt{\,\cdot\,}$ as $t\to 1$.

\subsection{Grand canonical ensemble}

Computations on $\mathfrak S_n$ can turn out to be difficult and many formulas can not be used to study the behaviour as $n\to\infty$.
A possible solution to this problem is to adopt a suitable randomization.
This has been successfully introduced by \cite{Fri93} and used also by \cite{Bo11a} as a tool to investigate combinatorial structures, and later applied in many contexts. 
The main idea of randomization is to define a one-parameter family of probability measures on $\stackrel{.}{\cup}_{n\geq 1}\mathfrak S_n$ for which cycle counts turn out to be independent. 
Then one is able to study their behaviour more easily, and ultimately the parameter is tuned in such a way that randomized functionals are distributed as in the non-randomized context. 
Let us see how to apply this in our work. We define
\begin{align}
 G_\Theta(t) = \mathrm{exp} \bigl( g_\Theta(t)\bigr)
\end{align}
with $g_\Theta(t)$ as in \eqref{eq:def_g_theta}.
%
%
If $G_\Theta(t)$ is finite for some $t>0$, then for each $\sigma\in \mathfrak S_n$ let us define the probability measure
\begin{align}
\label{eq:def_Pt}
\mathbb P_t\left[\sigma\right]:=\frac{1}{G_\Theta(t)}\frac{t^{n}}{n!} \prod_{k=1}^{n} \vartheta_{k}^{C_k}.
\end{align}
Lemma~\ref{lemma:Polya} shows that $\mathbb P_t$ is indeed a probability measure on $\stackrel{.}{\cup}_{n\geq 1}\mathfrak S_n$.
The induced distribution on cycle counts $C_k$ can easily be determined. 
\begin{lemma} \label{lemma:indep_cycle}
Under $\Pbt{\,\cdot\,}$ the $C_k$'s are independent and Poisson distributed with $$\Et{C_k} = \frac{\vartheta_k}{k}t^k.$$
\end{lemma}
\begin{proof}
From P\'olya's enumeration theorem (Lemma~\ref{lemma:Polya}) we obtain
\begin{eqnarray*}
\Et{\exp{-s C_k}}
&=&\sum_{n\geq 0}\sum_{\sigma\in \mathfrak S_n}\exp{-s C_k}\Pbt{\sigma}
=
\frac{1}{G_\Theta(t)}\sum_{n\geq 0}\sum_{\sigma\in \mathfrak S_n}\frac{t^n}{n!}(\vth_k\exp{-s })^{C_k}\prod_{j\leq n\atop j\neq k}(\vth_j)^{C_j}\\
&=&
\frac{1}{G_\Theta(t)}\mathrm{exp}\left( \sum_{j=0}^{+\infty}\frac{\vth_j}{j}t^j\right)\mathrm{exp}\left(\left(\exp{-s}-1\right)\frac{\vth_k}{k}t^k \right)\\
&=&
\mathrm{exp}\left(\left(\exp{-s}-1\right)\frac{\vth_k}{k}t^k \right).
\end{eqnarray*}
Analogously one proves the pairwise independence of cycle counts.
\end{proof}
Obviously the following conditioning relation holds:
$$
\mathbb P_t\left[\,\cdot\,|\,\mathfrak S_n\right]=\Pbn{\,\cdot\,}.
$$
A proof of this fact is easy and can be found for instance in \cite[Equation (1)]{Ha90}. We note that $w_n(x)$ is $\mathbb P_t$-a.s.~finite, since $\Et{w_n(x)}<+\infty$. Now since the conditioning relation holds for all $t$ with $G_\Theta(t)<+\infty$, one can try to look for $t$ satisfying ``$\Pbn{\,\cdot\,}\approx \Pbt{\,\cdot\,} $'', 
which heuristically means that we choose a parameter for which permutations on $\mathfrak S_n$ weigh as most of the mass of the measure $\mathbb P_t$. 
We have on $\mathfrak{S}_n$
$$ 
n = \sum_{j=1}^\ell \la_j = \sum_{k=1}^n k C_k.
$$
A natural choice for $t$ is thus the solution of 
\begin{align}
\label{eq:def_t_random}
 n 
= 
\mathbb E_t\left[\sum_{k=1}^\infty k C_k \right] 
=
\sum_{k=1}^\infty \vartheta_k t^k .
\end{align}
which is guaranteed to exist if the series on the 
right-hand side is divergent at the radius of convergence (we will see this holds true for our particular choice of weights).
We write $t= e^{-v_n}$ and use Lemma~\ref{lemma:polylog_asymp} in our case 
$\vartheta_k= {\left(\log k\right)^j}\frac{k^\alpha+\O{k^\beta}}{\Gamma(\alpha+1)}$ to obtain
\begin{eqnarray}
n 
&\stackrel{!}{=} &
{\frac{v_n^{-\alpha-1}}{\Gamma(\alpha+1)}
\left(\frac{\partial}{\partial \alpha}-\log v_n\right)^j\Gamma(\alpha+1)+\O{1} } \nonumber\\
&= & 
 \frac{v_n^{-\alpha-1}}{\Gamma(\alpha+1)}\left(\frac{\partial}{\partial \alpha}-\log v_n\right)^j\Gamma(\alpha+1)+\O{1}\nonumber \\ 
 & \Longrightarrow & 
v_n =\left(\frac{\left(\log n\right)^j}{n}\right)^{1/(\alpha+1)}+o\left(\frac{n}{\left(\log n\right)^j}\right)\label{eq:def_t_random_2}
\end{eqnarray}
We will fix this choice for the rest of the section.

\subsection{Limit shape and mod-convergence} 
In order to derive our main results from the measure $\mathbb P_t$ we will use a tool developed by \cite{KowNik10}, the \emph{mod-Poisson convergence}.
\begin{definition} 
A sequence of random variables $(Z_n)_{n\in \N}$ 
converges in the mod-Poisson sense with parameters 
{$(\mu_n)_{n\in \N}$} if the following limit
$$\lim_{n\to+\infty}{\mathrm{exp}(\mu_n (1-\exp{iu}))}\E{\exp{iu Z_n}}=\Phi(u)$$
exists for every $u \in \R$, and the convergence is locally uniform. The limiting function $\Phi$ is then continuous and $\Phi(0) = 1$.
\end{definition}
This type of convergence gives stronger results than a central limit theorem, indeed it implies a CLT 
(and other properties we will see below). For the rest of the Section let us fix $n^*$ and $\overline n$ as in \eqref{eq:def_n*}.
We obtain
\begin{proposition}\label{prop:lim_sha_rand}
Let $x\geq 0$ be arbitrary and $x^*:= x n^*$.
Furthermore, let $t=e^{-v_n}$ with $v_n$ as in \eqref{eq:def_t_random_2}.
Then the random variables $(w_n(x^*))_{n\in \N}$ converge in 
the mod-Poisson sense with parameters 
$$\mu_n = \overline n\, w_\infty^{\mathbf r}(x)(1+\o{1}),$$ where 
\begin{align}\label{eq:mathbf_r}
 w_\infty^{\mathbf r}(x):= \frac{\Gamma(\alpha,\,x)}{\Gamma(\alpha+1)}.
\end{align}
$\Gamma(\alpha,\,x)$ is the upper incomplete Gamma function.
\end{proposition}
\begin{proof} We have
\begin{align}
\label{eq:laplace_randomised_w_n}
 \Et{\exp{is w_n(x^*)}}=\Et{\exp{is \sum_{\ell=\floor{x^*}}^{\infty}C_\ell}}
=
\mathrm{exp}\left({\left(e^{is}-1\right)\sum_{\ell= \floor{x^*}}^{\infty}\frac{\vartheta_\ell}{\ell}{t^\ell}}\right).
\end{align}
This is the characteristic function of Poisson distribution. We thus obviously have
mod-Poisson convergence with limiting function $\Phi(t)\equiv 1$. It remains to compute the parameter $\mu_n$.
Applying Lemma~\ref{lemma:polylog_asymp} for $x=0$ and Lemma~\ref{lemma:approxsums} for $x>0$ together with \eqref{eq:def_t_random_2} gives
\begin{eqnarray}
&&\sum_{\ell=\floor{x^*}}^{+\infty}\frac{{(\left(\log \ell\right)^j)}\ell^{\alpha-1}+\O{\ell^{\beta-1}}}{\Gamma(\alpha+1)}t^\ell
\nonumber \\
&&=\frac1{\Gamma(\alpha+1)}
\left(n^*\right)^\alpha
\left(\frac{\partial}{\partial \alpha}-\log n^* \right)^j
\left({\Gamma(\alpha,x)}+\o{1}\right) \nonumber\\
&&=\frac1{\Gamma(\alpha+1)}
\left(n^*\right)^\alpha
\left(\log n^* \right)^j
\left({\Gamma(\alpha,x)}+\o{(n^*)^{\alpha}\left(\log n^* \right)^j}\right) \nonumber\\
&&=\frac{\Gamma(\alpha,x)}{\Gamma(\alpha+1)} \overline n +\o{\bar n}
\end{eqnarray}
We deduce that $\mu_n:=\overline n\, w_\infty^{\mathbf r}(x) +\o{\overline n}$.
This completes the proof.
\end{proof}

This yields a number of interesting consequences. In first place we can prove a CLT and detect the limit shape accordingly.
\begin{corollary}[CLT and limit shape for randomization]\label{corol:CLT_random}
With the notation as above, we have as $n\to\infty$ with respect to $\mathbb{P}_t$
\begin{eqnarray}
\label{eq:CLT_Pt}
&& \widetilde{w}_n^{\mathbf r}(x):=\frac{w_n(x^*)-
\overline n w_\infty^{\mathbf r}(x)}{\sqrt{\bar n } }
\stackrel{\mathcal L}{\to} 
\mathcal N(0,w_\infty^{\mathbf r}(x)).
\end{eqnarray}
Furthermore the limit shape of $w_n(x)$ is given by 
$w_\infty^{\mathbf r}(x)$ 
{(with scaling $A_n = \overline n $ 
and $B_n = n^*$, see \eqref{eq:def_limit_shape})}.
In particular, we can choose $\delta =0$ in \eqref{eq:def_limit_shape}.
%
\end{corollary}

\begin{proof}
The CLT follows immediately from \cite[Prop. 2.4]{KowNik10}, but also can be deduced easily from \eqref{eq:laplace_randomised_w_n} by replacing
$s$ by {$s (\bar n)^{-1/2}$}. It is also straightforward to show that $w_\infty^{\mathbf r}(x)$ is the limit shape.
For a given $\eps>0$, we choose $0=x_0 <x_1<\dots<x_\ell$ such that $w_\infty^{\mathbf r}(x_{j+1})-w_\infty^{\mathbf r}(x_{j})<\eps/2$ for $1\leq j \leq \ell-1$ and $w_\infty^{\mathbf r}(x_{\ell})<\eps/2$.
It is now easy to see that for each $x\in\R^+$
\begin{align*}
|(\overline n)^{-1}w_n(x^*)-w_\infty^{\mathbf r}(x)|>\eps \Longrightarrow \exists j \text{ with }|(\overline n)^{-1}w_n(x_j^*)-w_\infty^{\mathbf r}(x_j)|>\eps/2.
\end{align*}
Thus
\begin{align}
\label{eq:proof_limit_shape_Pt}
  \Pbt{\sup_{x\geq 0} |(\overline n)^{-1}w_n (x^*)-  w_\infty^{\mathbf r}(x)| \geq \eps}  
&\leq 
\sum_{j=1}^\ell \Pbt{|(\overline n)^{-1} w_n (x_j^*)- w_\infty^{\mathbf r}(x_j)| \geq \eps/2 }  
\end{align}
It now follows from \eqref{eq:CLT_Pt} that each summand in \eqref{eq:proof_limit_shape_Pt} tends to $0$ as $n\to \infty$. 
This completes the proof.
\end{proof}

Another by-product of mod-Poisson convergence of a sequence $(Z_n)_{n\in \N}$ is that one can approximate $Z_n$ 
with a Poisson random variable with parameter $\mu_n$, see \cite[Prop. 2.5]{KowNik10}.
However in our situation this is trivial since $w_n(x^*)$ is already Poisson distributed. 

%
%
%

As we are going to do in the next section, we are also interested in the (joint) behaviour of increments. 

\begin{proposition} 
\label{prop:marignals_randimised}
For all $x,\,y \in \R$, $y>x$, set
$$w_n(x,y):= w_n(x)-w_n(y)\ \text{ and } \ w_\infty^{\mathbf r}(x,y):=\frac{\Gamma(\alpha,x)-\Gamma(\alpha,y)}{\Gamma(\alpha+1)}.$$
Then
\begin{align}
\label{eq:w_n_x_y_weak_convergence}
\widetilde{w}_n^{\mathbf r}(x,y):=\frac{w_n(x^*,y^*)-{\bar n }w_\infty^{\mathbf r}(x,y)}{{\left(\bar n \right)^{1/2}}\sqrt{ w_\infty^{\mathbf r}(x,y)}}
\stackrel{\mathcal L}{\to} \mathcal N(0,1)
 \end{align}
as $n\to \infty$ with $x^* := x n^*$ and with $y^* := yn^*$.
%

Furthermore, $\widetilde{w}_n^{\mathbf r}(x)$ and $\widetilde{w}_n^{\mathbf r}(x,y)$ are asymptotically independent.
\end{proposition}
\begin{remark}
 As we will see, the proof of independence relies on the independence of cycles coming from Lemma 
 \ref{lemma:indep_cycle}. Therefore it is easy to generalize the above result to more than 
 two points.
\end{remark}

\begin{proof}
The proof of \eqref{eq:w_n_x_y_weak_convergence} almost the same as the proof of \eqref{eq:CLT_Pt} and we thus omit it.
Since
$$
w_n(x,y) = \sum_{k= x^*}^{y^*-1} C_k \ \text{ and } \ w_n(y) = \sum_{k= y^*}^{\infty} C_k
$$
and all $C_k$ are independent, we have that $\widetilde{w}_n^{\mathbf r}(x)$ and $\widetilde{w}_n^{\mathbf r}(x,y)$ are independent for each $n\in\N$.
Thus $\widetilde{w}_n^{\mathbf r}(x)$ and $\widetilde{w}_n^{\mathbf r}(x,y)$ are also independent in the limit.
\end{proof}
%
%
\subsection{Functional CLT}\label{sec:CLT_randomization}
The topic of this section is to prove a functional CLT for the profile $w_n(x)$ of the Young diagram. 
Similar results were obtained in a different framework by \cite{Ha90,DePi85} on the number of cycle counts not exceeding $n^{\lfloor x\rfloor}$, 
and by \cite{BelBoutEnri10} for Young diagrams confined in a rectangular box. We show
\begin{theorem}
\label{thm:func_CLT_pt}
 The process $\widetilde{w}_n^{\mathbf r}:\R^+\to\R$ \textup{(}see \eqref{eq:CLT_Pt}\textup{)} converges weakly with respect to $\mathbb{P}_t$ as $n\to\infty$ to a 
continuous process $\widetilde{w}_\infty^{\mathbf r}:\R^+\to\R$ with $\widetilde{w}_\infty^{\mathbf r}(x)\sim \mathcal N(0,\sigma_\infty^{\mathbf r}(x))$ and independent increments. 
\end{theorem}
The technique we will exploit is quite standardized (see \cite{Ha90}). We remark that, unlike in this paper where the Ewens measure is considered, we do not obtain here a Brownian process, as the variance 
of $\widetilde{w}_\infty^{\mathbf r}(t)-\widetilde{w}_\infty^{\mathbf r}(s)$ for $r\geq s$ is more complicated than in the case of the Wiener measure.

We know from Proposition~\ref{prop:marignals_randimised} the finite dimensional marginals of the process. More specifically
we have for $x_\ell\geq x_{\ell-1}\geq \dots\geq x_1\geq 0$ that
\begin{align}
\label{eq:marginals_w_n_Pt}
(\overline n)^{-1/2})\bigl(w_n(x^*_\ell),\,w_n(x^*_{\ell-1})-w_n(x^*_\ell),\,\ldots,\,w_n(x^*_1)-w_n(x^*_2)\bigr)\sim \mathcal N\left(\mathbf 0, \Sigma' \right)
\end{align}
where $\Sigma'$ is a diagonal matrix with 
$$\Sigma'_{11} = w_\infty^{\mathbf r}(x_\ell) \ \text{ and } \ \Sigma'_{jj} = w_\infty^{\mathbf r}(x_{\ell-j+1},x_{\ell-j+2}) \text{ for }j\geq 2.
$$
Now all we need to show to complete the proof of Theorem~\ref{thm:func_CLT_pt} is the tightness of the process $\widetilde{w}_n^{\mathbf r}$. 
In order to do so, we will proceed similarly to \cite{Ha90}, namely we will show that 
\begin{lemma}
\label{lemma:tightness} 
We have for $0\leq x_1<x\leq x_2<K$ with $K$ arbitrary 
\begin{equation}\label{eq:goal}
\Et{(\widetilde w_n^{\mathbf r}(x^*)-\widetilde w_n^{\mathbf r}(x_1^*))^2 (\widetilde w_n^{\mathbf r}(x_2^*)-\widetilde w_n^{\mathbf r}(x^*))^2}=\O{(x_2-x_1)^2}
\end{equation}
with  {$x^*:={x n^*}$, $x_1^*:={x_1 n^*}$ and $x_2^*:={x_2 n^*}$}.
\end{lemma}

Lemma~\ref{lemma:tightness} together with \cite[Theorem 15.6]{Bi99} implies that the process $\widetilde{w}_n^{\mathbf r}$ is tight.
This and the marginals in \eqref{eq:marginals_w_n_Pt} prove Theorem~\ref{thm:func_CLT_pt}.
\begin{proof}[Proof of Lemma \ref{lemma:tightness}]
We define
\begin{equation}\label{eq:goal_E}
E^\ast:
=
\Et{(\widetilde w_n^{\mathbf r}(x^\ast)-\widetilde w_n^{\mathbf r}(x_1^\ast))^2 (\widetilde w_n^{\mathbf r}(x_2^\ast)-\widetilde w_n^{\mathbf r}(x^\ast))^2}.
\end{equation}
The independence of the cycle counts leads us to
\begin{eqnarray}
E^\ast\label{eq:center}
&=&
\left(\sum_{k=x_1^\ast}^{x^\ast-1}{(\bar n)^{-1}}\frac{\theta_k}{k}t^k \right) \cdot \left(\sum_{k=x^\ast}^{x_2^\ast-1}{(\bar n)^{-1}}\frac{\theta_k}{k}t^k\right)\nonumber\\
&\stackrel{Lem.~\ref{lemma:approxsums}}{\sim}&
\left(\frac{{(\bar n)^{-1}}}{\Gamma(\alpha+1)}\int_{x_1^\ast}^{x^\ast}{(\log t)^j} t^{\alpha-1}\exp{-t}\mathrm d t\right) 
\left(\frac{{(\bar n)^{-1}}}{\Gamma(\alpha+1)}\int_{x^\ast}^{x_2^\ast}{(\log t)^j}t^{\alpha-1}\exp{-t}\mathrm d t\right)\nonumber\\
&=&
\left(\frac{{\left(\bar n\right)^{-1}\left(\frac{\bar n}{n}\right)^{-\alpha}\left(-\log \left(\frac{\bar n}{n}\right)\right)^j}\left(\Gamma(\alpha,\,x_1)-\Gamma(\alpha,\,x)\right)}{\Gamma(\alpha+1)}\right)\nonumber\\
&&\left(\frac{{\left(\bar n\right)^{-1}\left(\frac{\bar n}{n}\right)^{-\alpha}\left(-\log \left(\frac{\bar n}{n}\right)\right)^j}\left(\Gamma(\alpha,\,x)-\Gamma(\alpha,\,x_2)\right)}{\Gamma(\alpha+1)}\right)+\o{1}\nonumber\\
&=&
\O{(x-x_1)(x_2-x)}=\O{(x_2-x_1)^2}.\nonumber
\end{eqnarray}
Here we have used the fact that $\Gamma(\alpha,\,\cdot)$ is a Lipschitz function and the assumption that $x_1<x\leq x_2<K$. Also note that $\left(\bar n\right)^{-1}\left(\frac{\bar n}{n}\right)^{-\alpha}(-\log \left(\frac{\bar n}{n}\right))^j =\O{1}.$
\end{proof}

\section{Saddle point method}
\label{sec:estmgf}

The aim of this section is to study the asymptotic behaviour of $w_n(x)$ with respect to $\Pbn{\cdot}$ as $n\to\infty$
and to compare the results with the results in Section~\ref{sec:randomization}.

%
%
%
%
%
There are at least two approaches with which to tackle this problem: one is more probabilistic and was employed by \cite{ErGr08} in their paper. The second one was 
first developed in \cite{MaNiZe11} from the standard saddle point method.\\
The first method to study the asymptotic statistics of $w_n(x)$ with respect to $\Pbn{\cdot}$ as $n\to\infty$
is the so called Khintchine method.
We illustrate this method briefly with the normalisation constant $h_n$ (see \eqref{eq:def_weighted_pb_la}).
The first step is to write down a Khintchine's type representation for the desired quantity. For $h_n$ this is given by
\begin{align}
\label{eq:h_n_Khintchine}
 h_n 
 =
t^{-n}\mathrm{exp} \left(\sum_{k=1}^n \frac{\vartheta_k}{k} t^k \right) \Pbt{\sum_{k=1}^n k C_k = n}
\end{align}
with $t>0$ and $\Pbt{\,\cdot\,}$ as in Section~\ref{sec:randomization}. The second step is to choose the free parameter $t$ 
in such a way that $\Pbt{\sum_{k=1}^n k C_k = n}$ gets large. Here one can choose $t$ to be the solution of
the equation $\sum_{k=1}^n \vartheta_k t^k =n$.

This argumentation is very close to the argumentation relying on complex analysis and generating functions. Indeed, it is easy to see that
\eqref{eq:h_n_Khintchine} is equivalent to 
\begin{align}
\label{eq:gnerating_h_n}
h_n 
=
[t^n]\left[ \mathrm{exp}\left(g_\Theta(t)\right) \right]
\end{align}
with $g_\Theta(t)$ as in \eqref{eq:def_g_theta}. Furthermore, the choice of $t$ is (almost) 
the solution of the saddle point equation $tg_\Theta'(t) =n$.
We have of course to justify \eqref{eq:gnerating_h_n} (or \eqref{eq:h_n_Khintchine}).
But this follows immediately from the definition of $h_n$ and Lemma~\ref{lemma:Polya}.

We prefer at this point to work with the second approach.
We begin by writing down the generating functions of the quantities we would like to study.

\begin{lemma}
\label{lem:generating_w_n}
We have for $x \geq 0$ and $s \in\R$
\begin{align}
\label{eq:generating_w_n}
\En{\mathrm{exp}\bigl(-s w_n(x)\bigr)}
= 
\frac{1}{h_n}
[t^n] \left[\mathrm{exp}\left( g_\Theta(t) + (e^{-s}-1)\sum_{k=\floor{x} }^\infty \frac{\vartheta_k}{k}t^k \right) \right].
\end{align}
\end{lemma}
\begin{remark}\label{rem:s_positive}
Although the expressions in Lemmas~\ref{lem:generating_w_n} and \ref{lem:generating_w_n_x_y} 
hold in broader generality, starting from Subsection~\ref{subsec:log-n-adm} we will calculate moment generating functions on the positive half-line, namely we can assume all parameters $ s_1,\,\ldots,\,s_\ell$ etc to be non-negative, 
according to \cite[Theorem 2.2]{Cha07}.
\end{remark}

\begin{proof}
 It follows from the definitions of $\Pbn{\,\cdot\,}$ and $w_n(x)$ (see \eqref{eq:def_weighted_pb_ck}) that
\begin{align}
 h_n \En{\mathrm{exp}\bigl(-s w_n(x)\bigr)}
&=
\frac{1}{n!} \sum_{\sigma\in\mathfrak S_n} \mathrm{exp}\left(-s \sum_{k=\floor{x}}^n C_k\right) \prod_{k=1}^{n} \vartheta_k^{C_k} \\
&=
\frac{1}{n!} \sum_{\sigma\in\mathfrak S_n} \prod_{k=1}^{\floor{x}-1} \vartheta_k^{C_k} \prod_{m=\floor{x}}^\infty (\vartheta_k e^{-s})^{C_k} \nonumber
\end{align}
Applying now Lemma~\ref{lemma:Polya}, we obtain 
\begin{align}
\sum_{n=0}^\infty \frac{t^n}{n!} h_n \En{\mathrm{exp}\bigl(-s w_n(x)\bigr)} 
&=
\mathrm{exp}\left(\sum_{k=1}^{\floor{x}-1} \frac{\vartheta_k}{k}t^k + e^{-s} \sum_{k=\floor{x}}^\infty \frac{\vartheta_k}{k}t^k \right)\\
&=
\mathrm{exp}\left( g_\Theta(t) + (e^{-s}-1)\sum_{k=\floor{x} }^\infty \frac{\vartheta_k}{k}t^k \right) 
\end{align}
Equation \eqref{eq:generating_w_n} now follows by taking $[t^n]$ on both sides.
\end{proof}

We are also interested in the joint behaviour at different points of the limit shape.
The results in Section~\ref{sec:randomization} suggest that the increments of $w_n(x_{j+1})-w_n(x_j)$ are independent
for $x_\ell\geq x_{\ell-1}\geq \dots\geq x_1\geq 0$.
It is thus natural to consider
\begin{equation}\label{eq:increm}
\mathbf{w}_n(\textbf{x})=\bigl(w_n(x_\ell),\,w_n(x^*_{\ell-1})-w_n(x_\ell),\,\ldots,\,w_n(x_1)-w_n(x_2)\bigr).
\end{equation}
We obtain
\begin{lemma}
\label{lem:generating_w_n_x_y}
We have for $\textbf{x}=(x_1,\dots,x_\ell)\in\R^\ell$  with $x_\ell\geq x_{\ell-1}\geq \dots\geq x_1\geq 0$ and $\textbf{s}=(s_1,\dots,s_\ell)\in\R^\ell$ 
\begin{align}
\label{eq:generating_w_n_x_y}
\En{\mathrm{exp}\bigl(-\langle \textbf{s},\textbf{w}_n(\textbf{x})\rangle \bigr)}
= 
\frac{1}{h_n}
[t^n] \left[\mathrm{exp}\left( g_\Theta(t) +\sum_{j=1}^\ell (e^{-s_j}-1)\sum_{k=\floor{x_j} }^{\floor{x_{j+1}-1} } \frac{\vartheta_k}{k}t^k \right) \right]
\end{align}
\end{lemma}
with the convention $x_{\ell+1}:=+\infty$.
The proof of this lemma is almost the same as for Lemma~\ref{lem:generating_w_n} and we thus omit it.


\subsection{\texorpdfstring{Log-$n$-admissibility}{Log-n-admissibility}}\label{subsec:log-n-adm}
The approach with which we first addressed the study of the limit shape is derived from the saddle point method for approximating integrals 
in the complex plane. We want to introduce the definition of log-$n$-admissible function, generalizing the analogous concept introduced in \cite{MaNiZe11}. We stress that here, in comparison to the definition of log- (or equivalently Hayman) admissibility used there, we consider a family of functions parametrized by $n$ for which log-admissibility holds simultaneously. The definition is therefore a natural extension.
\begin{definition} 
\label{def:admissible}

Let $\bigl( g_n(t)\bigr)_{n\in\N}$ with $g_n(t) = \sum_{k=0}^\infty g_{k,n} t^k$ be given with radius of convergence $\rho > 0$ and $g_{k,n} \geq 0$.
We say that $\bigl( g_n(t)\bigr)_{n\in\N}$ is \emph{$\log$-$n$-admissible} 
if there exist functions $a_n,b_n:[0,\rho) \to \R^+$, $R_n : [0,\rho) \times (-\pi/2, \pi/2) \to \R^+$ and a sequence $(\delta_n)_{n\in\N}$ s. t.
\begin{description}
  \item[Saddle-point] For each $n$ there exists $r_n \in [0,\rho)$ with \begin{equation}a_n(r_n)= n \label{eq:saddle_point_equation} \end{equation}
  \item[Approximation] For all $\abs{\varphi} \leq \delta_n$ we have the expansion
        \begin{align}
          g_n(r_ne^{i\varphi})
          =
          g_n(r_n) + i\varphi a_n(r_n)-\frac{\varphi^2}{2} b_n(r_n)
          + R_n(r_n,\varphi)
          \label{eq_admissible_expansion}
        \end{align}
        where $R_n(r_n,\varphi) = o(\varphi^3 \delta_n^{-3})$. 
  \item[Divergence] $b_n(r_n) \to \infty$ and $\delta_n \to 0$ as $n \to \infty$.
  \item[Width of convergence] We have $\delta_n^2 b_n(r_n) - \log b_n(r_n) \to +\infty$ as $n \to + \infty$.
  \item[Monotonicity] For all $\abs{\varphi} > \delta_n$, we have
        \begin{align}\label{eq:monotonicity}
            \Re{g_n(r_n e^{\i \varphi})} \leq \Re{g(r_n e^{\pm \i \delta_n})}.
        \end{align}
\end{description}
\end{definition}

The approximation condition allows us to compute the functions $a$ and $b$ exactly.  We have
\begin{align}
\label{eq:a_n_general_explicit}
  a_n(r) &= rg_n'(r), \\
  \label{eq:b_n_general_explicit}
  b_n(r) &= rg_n'(r) +r^2 g_n''(r)
\end{align}
Clearly $a_n$ and $b_n$ are strictly increasing real analytic functions in $[0, \rho)$. 
The error in the approximation can similarly be bounded, so that
\[
  R_n(r,\varphi)= \varphi^3 \O{ rg_n'(r) + 3r^2 g_n''(r) + r^3 g_n'''(r)}
\]

Having proved Lemma \ref{lem:generating_w_n} we are now able to write down in a more explicit way generating functions. What we are left with is trying to extract the coefficients 
of the expansion given therein. This is the content of
\begin{theorem} \label{thm:generalasymptotic}
Let $\bigl( g_n(t)\bigr)_{n\in\N}$ be $\log$-$n$-admissible with associated functions $a_n$, $b_n$ and constants $r_n$. 
Call 
$$G_n:=[t^n]\exp{g_n(t)}.$$
Then $G_{n}$ has the asymptotic expansion
\begin{equation}\label{eq:G_n}
 G_{n} = \frac{1}{\sqrt{2 \pi}} (r_{n})^{-n} b_n(r_{n})^{-1/2} \exp{g_n(r_{n})} (1 + o(1)).
\end{equation}
\end{theorem}
\begin{remark}
As it is explained in 
\cite[Chapter VIII]{FlSe09} it is possible to take into account more error terms in the expansion of $g_n$. We could also extract here the behaviour of the coefficients $h_n$. However in the computations we will not need it explicitly, since these terms will always cancel out.
\end{remark}

\begin{proof}[Proof of Theorem \ref{thm:generalasymptotic}]
The proof is exactly the same as in \cite[Prop. 2.2]{MaNiZe11} and we thus give only a quick sketch of it, referring 
the reader to this paper for more details. As in the well-known 
saddle point method, we want to evaluate the integral
$$
\frac{1}{2\pi \i}\oint_\gamma \mathrm{exp}\left(g_n(z)\right)\frac{\mathrm d z}{z^{n+1}}.
$$
We choose as contour the circle $\gamma:= r_n e^{\i\varphi}$ with $\varphi\in[-\pi,\pi]$. 
On $\vr \in [-\delta_n,\,\delta_n]$ after changing to polar coordinates we can expand the function $g$ as
$$
\int_{-\delta_n}^{\delta_n}\mathrm{exp}\left(g_n(r) + \i\varphi a_n(r)-\frac{\varphi^2}{2} b_n(r)
          + o(\varphi^3 \delta_n^{-3})-\i n\vr\right)\mathrm d \vr
$$
We now choose $r_n$ such that $a(r_n)=r_n g'_n(r_n)=n$ in order to cancel the linear terms in $n$. 
This allows us to approximate the integral on the minor arc with a Gaussian. One shows that away from the saddle point 
(so for $\abs{\vr}>\delta_n$) the 
contribution is exponentially smaller than on the minor arc and thus it can be neglected.
\end{proof}
We would like to emphasize also that it will be not 
always possible to solve the saddle point equation \eqref{eq:saddle_point_equation} exactly. However it is enough to find an $r_n$ such that
\begin{equation}\label{eq:approximate_saddle_equation}
a(r_n)-n=\o{\sqrt{b(r_n)}}
\end{equation}
holds.

\subsection{Calculation of the limit shape}\label{sec:calculation}
In this section we will derive the limit shape for Young diagrams for the class of measures given by the weights. We will not go into all the details to prove the log-$n$-admissibility for the most general case,
but will try to give a precise overview of the main steps nonetheless. One important remark we have to make is that our parameter $s$ 
will not be fixed, but will be scaled and hence dependent on $n$. 
This comes from the fact that for a fixed $s$  \eqref{eq:saddle_point_equation} becomes a fixed point equation whose solution cannot be given constructively, but has only an implicit form. We were not able to use this information for our purposes, and hence preferred to exploit a less general, but more explicit parameter to calculate asymptotics. \\
\subsubsection{Limit shape}
The main goal of this subsection is to prove that the weights (\ref{eq:weights}) induce a sequence of log-$n$-admissible functions of which we can recover 
the asymptotics of $g_n(r_n)$. This will give us the limit shape of the Young diagram according to Theorem \ref{thm:generalasymptotic}. Hence we pass to showing Theorem~\ref{thm:limit_shape_saddle}, that is the limit shape is 
 $$
 w_\infty^{\mathbf s}(x)=\frac{\Gamma(\alpha,x)}{\Gamma(\alpha+1)},
 $$
 and the fluctuations at a point $x$ behave like
 \begin{eqnarray*}&&
\widetilde w_n^{\mathbf s}(x)
=
\frac{w_n(x^*)- \overline n\left(w_\infty^{\mathbf s}(x)+z_n^{\mathbf s}\right)}{(\overline n)^{1/2}}
\stackrel{\mathcal L}{\longrightarrow}
\mathcal N\left(0, \sigma_\infty^2(x)\right)
\end{eqnarray*}
 with $x^\ast= xn^\ast$ and $\sigma_\infty^2(x)$ as in Theorem~\ref{thm:limit_shape_saddle}.

\begin{remark}
We note that the limit shape matches the one obtained in \cite[Thm. 4.8]{ErGr08} and also the one obtained in the present paper in the randomized case (cf. the definition of $w_\infty^{\mathbf r}(x)$ of Prop. \ref{prop:lim_sha_rand}).
\end{remark}

As indicated in the introduction, we proof Theorem~\ref{thm:limit_shape_saddle} by computing the Laplace transform.
We now have
 \begin{proposition}\label{prop:relies2}
We have for $s=O(1)$ and with respect to $\mathbb{P}_n$ as $n\to\infty$
$$
\En{\mathrm{exp}\bigl(-s\widetilde w_n^{\mathbf s}(x)\bigr)}={\sigma_\infty^2(x)}\frac{s^2}{2}+\O{(\overline n)^{-\frac{1}{2}}s^3}. 
$$
\end{proposition}
Obviously, Proposition~\ref{prop:relies2} immediately implies Theorem~\ref{thm:limit_shape_saddle}. Moreover, knowing the behaviour 
of the Laplace transform enables us to compute the asymptotics of the Young diagram in the limit. 
 More precisely, we now can proof the large deviation estimates in Prop.~\ref{prop:forall}.
\begin{proof}[Proof of Prop.~\ref{prop:forall}]
The strategy we adopt was first exploited in \cite[Theorem 4.1]{MaNiZe11}. Specifically, let $\sigma_n$ be the limit variance as in Thm.~\ref{thm:limit_shape_saddle}. 
Define the normalized log-moment generating function as \eqref{eq:cum_generating_func}. One bounds then the probability 
$\mathbb P\left(|\widetilde w_n(x^*)-a|\leq\epsilon\right)$ by a random variable $Y$ of mean $a$ and using the second moment method.
We omit the details since the computations are almost the same as in \cite{MaNiZe11}.
\end{proof}
We can also determine the behaviour of the increments of the function $w_n(\cdot)$. 
\begin{theorem}
\label{thm:beh_increments_saddle}
For $\ell\geq 2$ and $x_\ell\geq x_{\ell-1}\geq \dots\geq x_1\geq 0$, let 
$$\widetilde{\mathbf{w}}_n^{\mathbf s}(\mathbf x)=\bigl(\widetilde w_n^{\mathbf s}(x_\ell),\,\widetilde w_n^{\mathbf s}(x_{\ell-1})-\widetilde w_n^{\mathbf s}(x_\ell),\,\ldots,\,\widetilde w_n^{\mathbf s}(x_1)-\widetilde w_n^{\mathbf s}(x_2)\bigr).$$ 
Set $x_{\ell+1}=+\infty$. For $1\leq j<i<\ell$ we have that 
\begin{align}
\widetilde w_\infty^{\mathbf s}(x_i,\,x_j)&:=\lim_{n\to+\infty}\mathrm{Cov}\left(\widetilde w_n^{\mathbf s}(x_j)-\widetilde w_n^{\mathbf s}(x_{j+1}),\,\widetilde w_n^{\mathbf s}(x_i)-\widetilde w_n^{\mathbf s}(x_{i+1})\right)\\
&=\frac{\left(\Gamma(\alpha+1,\,x_i)-\Gamma(\alpha+1,\,x_{i+1})\right)\left(\Gamma(\alpha+1,\,x_{j})-\Gamma(\alpha+1,\,x_{j+1})\right)}{\Gamma(\alpha+1)\Gamma(\alpha+2)}.
\nonumber
\end{align}
\end{theorem}
\begin{remark}
 Let us comment briefly on Thm.~\ref{thm:beh_increments_saddle}. What we obtained in this result is most unexpected: 
cycle counts are asymptotically independent under very mild assumptions (see Lemma~\ref{lemma:BetzUel}). 
The assumption of the lemma holds in our case as the growth of the parameters $\vth_n$ is algebraic.
The fact that the increments depend on disjoint sets of cycles would have 
suggested the asymptotic independence of $w_n(y^\ast)$ from $w_n(x^\ast)-w_n(y^\ast)$. We are aware of the work of \cite{BaMaZa07} handling this 
issue in the case of the Ewens 
sampling formula, in particular showing that partial 
sums of cycle counts need not converge to processes with independent increments. Our result extends this idea in the sense that it shows 
the explicit covariance matrix for 
a whole category of generating functions. It would be interesting to provide a heuristic explanation for this theorem.
\end{remark}
\subsubsection{\texorpdfstring{Log-$n$-admissibility}{Adm}}
In order to determine the limit shape we would like to prove the 
log-$n$-admissibility of the function explicited in \eqref{eq:generating_w_n}. 
To be more precise, what we have to prove is

\begin{lemma}\label{lemma:funct_admiss}
Let $\vth_k$ be as in \eqref{eq:weights},  $s,x\geq 0$ and $x^\ast=xn^\ast$, $s^\ast=s (\overline n)^{1/2}$ with $n^\ast,\overline n$ as in \eqref{eq:def_n*}. Then 
$$
g_\Theta(t) + (e^{-s^\ast}-1)\sum_{k=\floor{x^\ast} }^\infty \frac{\theta_k}{k}t^k 
$$
is log-$n$-admissible for all $x,s\geq 0$ with 
\begin{align}
\label{eq:r_n_adm} 
 r_n:=\exp{-v_n}, \,  v_n=p_n(1-sq_n) 
 \end{align}
 for some 
 \begin{align}
 \label{eq:p_n_adm} 
 p_n \sim (n^\ast)^{-1} 
 \ \text{ and  }\ 
 q_n \sim (n^\ast)^{-1/2} \frac{\Gamma(\alpha+1,\,x)}{\Gamma(\alpha+2)}.
\end{align}
\end{lemma}
Let us consider $p_n$ and $q_n$ in the case $\vth_k =\frac{(\log k )^jk^\alpha}{\Gamma(\alpha+1)}$ more explicitly. In this situation, $p_n$ is the solution of the equation
\begin{align}
\label{eq:p_n_special}
 \sum_{k=1}^\infty  \frac{(\log k)^j k^\alpha}{\Gamma(\alpha+1)} e^{-kp_n} =n.
\end{align}
Using Lemma~\ref{lemma:polylog_asymp}, one can easily show that  $p_n \sim (n^\ast)^{-1}$. Furthermore, we have
 \begin{align}
 \label{eq:qn_precise}
  q_n
 =
\frac{(np_n)^{-1/2}\left(\frac{\partial}{\partial \alpha} -\log p_n\right)^j\Gamma(\alpha+1,\,x)}
{(\alpha+1)\left(\frac{\partial}{\partial \delta} -\log p_n\right)^{j}\Gamma(\alpha+1) +j\left(\frac{\partial}{\partial \delta} -\log p_n\right)^{j-1}\Gamma(\alpha+1)}. 
 \end{align}
The lower order terms of $p_n$ and $q_n$ in \eqref{eq:p_n_special} and \eqref{eq:qn_precise} are important for the saddle point solution 
in the sense that the condition $a_n(r_n) = n +\o{\sqrt{b_n}}$ is fulfilled (see \eqref{eq:approximate_saddle_equation}).
However, for the computation of the limit shape, the fluctuations and the cumulants, we require only the leading term of $p_n$ and $q_n$.

We require in the proof of Lemma~\ref{lemma:funct_admiss} the observation 
 \begin{align}
 \label{eq:size_pn}
  p_n \sim  (n^*)^{-1}
  \text{ and } p_n^{-\alpha}(-\log p_n)^j\sim \overline n,
 \end{align}
which follows with a straightforward computation.
\begin{proof}[Proof of Lemma~\ref{lemma:funct_admiss}]
We now verify the conditions in Definition~\ref{def:admissible}. 

\begin{description}
\item[Saddle-point and approximation] \label{page}
We begin with the case $\beta=0$ and compute first the size of $b(r_n)$. 
We get with Lemma~\ref{lemma:polylog_asymp}, \eqref{eq:b_n_general_explicit} and \eqref{eq:size_pn}
 \begin{eqnarray}
b_n(r_n)
&=&
\sum_{k=1}^{+\infty}\frac{(\log k)^j k^{\alpha+1}}{\Gamma(\alpha+1)}\exp{-k v_n}+(\exp{-s^*}-1)\sum_{k=\floor{x^*}}^{+\infty}\frac{(\log k)^j k^{\alpha+1}}{\Gamma(\alpha+1)}\exp{-k v_n}\nonumber\\
&=&
\O{\sum_{k=1}^{+\infty}(\log k)^j k^{\alpha+1}\exp{-k v_n}}
=
\O{v_n^{-\alpha-2}\left(\log v_n\right)^j}.\label{eq:b_n_estimate}
\end{eqnarray} 
We thus have to show that $a_n(r_n) = n + \o{\left(v_n^{-\alpha-2}\left(\log v_n\right)^j\right)^{1/2}}$ to get a suitable saddle point solution, see \eqref{eq:approximate_saddle_equation}.
Using the asymptotic behaviour of $v_n$, we have to show $a_n(r_n) = n + \o{n^{\frac{\alpha+2}{2(\alpha+1)}}\left(\log n\right)^j}$.
We use \eqref{eq:a_n_general_explicit} and obtain
\begin{eqnarray}
a_n(r_n)
&=&
\sum_{k=1}^{+\infty}\frac{(\log k)^j k^\alpha}{\Gamma(\alpha+1)}\exp{-k v_n}+(\exp{-s^*}-1)\sum_{k=\floor{x^*}}^{+\infty}\frac{(\log k)^j k^\alpha}{\Gamma(\alpha+1)}\exp{-k v_n}.
\end{eqnarray} 
We begin with the first sum. We use Lemma~\ref{lemma:polylog_asymp} and get
\begin{eqnarray}
 & &\sum_{k=1}^{+\infty}\frac{(\log k)^j k^\alpha}{\Gamma(\alpha+1)}\exp{-k v_n}
 =
 \frac{v_n^{-\alpha-1}}{\Gamma(\alpha+1)}\left(\frac{\partial}{\partial \alpha}-\log v_n\right)^j\Gamma(\alpha+1)+\O{1} \\
 &=&
 \frac{\left(p_n\right)^{-\alpha-1}}{\Gamma(\alpha+1)(1-q_n)^{\alpha+1}}
 \cdot \left(\frac{\partial}{\partial \alpha}-\log p_n +\log(1-q_n)\right)^j\Gamma(\alpha+1) +O(1). \nonumber
 \end{eqnarray} 
  We choose at this point $p_n$ to be the solution of 
 \begin{align}
  n=
  \frac{\left(p_n\right)^{-\alpha-1}}{\Gamma(\alpha+1)(1-q_n)^{\alpha+1}}
 \cdot \left(\frac{\partial}{\partial \alpha}-\log p_n \right)^j\Gamma(\alpha+1).
 \end{align}
A straightforward computation gives $p_n \sim (n^\ast)^{-1}$.  We now expand $(1-q_n)^{-\alpha-1}$ and $\log(1-q_n)$. This then gives 

 \begin{eqnarray}
  &&\frac{p_n^{-\alpha-1}}{\Gamma(\alpha+1)}\left(\frac{\partial}{\partial \alpha}-\log p_n\right)^j\Gamma(\alpha+1) \nonumber\\
  &&+
   \frac{(\alpha+1) q_n \left(p_n\right)^{-\alpha-1}}{\Gamma(\alpha+1)}\left(\frac{\partial}{\partial \alpha}-\log \left(p_n\right)\right)^j\Gamma(\alpha+1) \nonumber\\
  &&+
  \frac{q_n p_n^{-\alpha-1}}{\Gamma(\alpha+1)}   j\left(\frac{\partial}{\partial \alpha}-\log p_n\right)^{j-1} \Gamma(\alpha+1)  +
  O\bigl( (\log p_n)^j p_n^{-\alpha-1} q_n^2  \bigr)\nonumber\\
&=&
\label{eq:saddle_comp_1}
n+
      \frac{(\alpha+1) q_n \left(p_n\right)^{-\alpha-1}}{\Gamma(\alpha+1)}\left(\frac{\partial}{\partial \alpha}-\log p_n\right)^j\Gamma(\alpha+1) \\
  &&+
  \frac{q_n \left(p_n\right)^{-\alpha-1}}{\Gamma(\alpha+1)}    j\left(\frac{\partial}{\partial \alpha}-\log \left(p_n\right)\right)^{j-1} \Gamma(\alpha+1)  +
  \O{n^{\frac{1}{1+\alpha}} (\log n)^{-\frac{j}{1+\alpha}}}.\nonumber
 \end{eqnarray}
Since $\frac{1}{1+\alpha}< \frac{2+\alpha}{2(1+\alpha)}$, we can ignore the big-$O$ term.
We now come to the second sum. We use Lemma~\ref{lemma:approxsums} with $z_n= xn^*$ and a similar estimate as in \eqref{eq:b_n_estimate} to obtain
\begin{align}
 &(\exp{-s^*}-1)\sum_{k=\floor{x^*}}^{+\infty}\frac{(\log k)^j k^\alpha}{\Gamma(\alpha+1)}\exp{-k v_n}\nonumber\\
 =&
 -s^\ast\sum_{k=\floor{x^*}}^{+\infty}\frac{(\log k)^j k^\alpha}{\Gamma(\alpha+1)}\exp{-k v_n} + \O{(s^\ast)^2 v_n^{-\alpha-1}\left(\log v_n\right)^j}
 \nonumber\\
 \label{eq:saddle_comp_2}
 =&
  -s^* \frac{\left(p_n\right)^{-\alpha-1}}{\Gamma(\alpha+1)}\left(\frac{\partial}{\partial \alpha} -\log \left(p_n\right) \right)^j \Gamma(\alpha+1,x)
+\O{n^{\frac{1}{1+\alpha}} (\log n)^{\frac{-j}{1+\alpha}}}.
\end{align}
We can ignore here also the big-$O$ term. 
We thus have to check that the remaining terms in \eqref{eq:saddle_comp_1} and \eqref{eq:saddle_comp_2}
combined give $a_n(r_n) = n + \o{\sqrt{b_n(r_n)}}$. This follows from the definition of $q_n$, see \eqref{eq:qn_precise}.

We now come to the case $\beta < \alpha/2$. We have
\begin{align}
&
\sum_{k=1}^{+\infty}\frac{(\log k)^j k^{\beta}}{\Gamma(\alpha+1)}\exp{-k v_n} + (\exp{-s^*}-1)\sum_{k=\floor{x^*}}^{+\infty}\frac{(\log k)^j k^{\beta}}{\Gamma(\alpha+1)}\exp{-k v_n}\nonumber\\
=&
\, \O{ \sum_{k=1}^{+\infty}\frac{(\log k)^j k^{\beta}}{\Gamma(\alpha+1)}\exp{-k v_n} }
=
\O{n^{\frac{\beta+1}{\alpha+1}} (\log n)^{\frac{1}{\alpha+1}}}
\end{align}
Since $\beta < \alpha/2$, we obtain again $a_n(r_n) = n + \o{\sqrt{b_n(r_n)}}$. This completes the proof of this point.
For completeness, we give in Remark~\ref{rem:general_weights} some hints how to adjust this proof to the case $\vth_m\sim (\log m)^j m^\alpha$.

\item[Divergence] By the above calculations we set $\delta_n:=(n^*)^{-\xi}$ with $\frac{\alpha+3}{3}<\xi<\frac{\alpha+2}{2}$. This position holds also in the case $\beta>0$.
\item[Monotonicity] 
In the region $\abs{\varphi}>\delta_n$ we wish to show \eqref{eq:monotonicity}. We distinguish between the cases $\varphi=\o{v_n}$, $\varphi\neq \o{v_n}$ and $\abs{\varphi}=\o{1}$, and finally $\abs{\varphi}>C$.
First remember that $g_n\left(r_n \exp{\pm \i \delta_n}\right)=\O{\left(\log n^*\right)^j(n^*)^\alpha}$ by Lemma \ref{lemma:polylog_asymp}. Thus here we have:
\begin{enumerate}
 \item if $\varphi=\o{v_n}$, then by a change of variable $t\leadsto (v_n-\i\varphi)t$
 \begin{eqnarray*}
 && \sum_{k\geq \floor{x^*}}\frac{(\log k)^jk^{\alpha-1}}{\Gamma(\alpha+1)} \exp{-k(v_n-\i \varphi)}\\
 &&\sim\frac{(v_n-\i \varphi)^{-\alpha}}{\Gamma(\alpha+1)} \int_{x}^{+\infty} (\log t)^j t^{\alpha-1}\exp{-t}\mathrm d t\\
&&\sim\left(\frac{\partial}{\partial \alpha}+\log n^*\right)^j\frac{\Gamma(\alpha,x)}{\Gamma(\alpha+1)} (v_n-\i \varphi)^{-\alpha}.
 \end{eqnarray*}
Considering the factor $\exp{-s^*}-1$ we obtain that the summand is negligible with respect to $\Re{g(r_n e^{\pm \i\delta_n})}$.
 \item If $\varphi\neq \o{v_n}$ but $\abs{\varphi}=\o{1}$, then
 \begin{eqnarray}
 && \sum_{k\geq \floor{x^*}}\frac{ (\log k)^jk^{\alpha-1}}{\Gamma(\alpha+1)} \exp{-k(v_n-\i \varphi)}\nonumber\\
 && \sim\frac{(v_n-\i \varphi)^{-\alpha}}{\Gamma(\alpha+1)} \int_{x-\i x\varphi n^*+\o{1}}^{+\infty} (\log k)^j t^{\alpha-1}\exp{-t}\mathrm d t\nonumber\\
&&\sim\left(\frac{\partial}{\partial \alpha}+\log n^*\right)^j\frac{\Gamma(\alpha,x-\i x\varphi n^*)}{\Gamma(\alpha+1)} 
  (v_n-\i \varphi)^{-\alpha}\label{eq:case_little_o}
 \end{eqnarray}
 and afterwards use the fact that $\Gamma(\alpha,x+\i y)=\O{y^{\alpha-1}}$ for $\abs{y}$ large. Hence the RHS of \eqref{eq:case_little_o} becomes
 $$
 \O{(n^*)^{\alpha-1}}
  (v_n-\i \varphi)^{-\alpha}=\O{(n^*)^{\alpha-1}\varphi^{-1}}
 $$
 As $\varphi\neq \o{v_n}$, we obtain that $\O{(n^*)^{\alpha-1}\varphi^{-1}}=\O{(n^*)^\alpha}\o{1}$ which is enough to show \eqref{eq:monotonicity} in this region.
\item To conclude we consider the case $\abs{\varphi}>C$: the function $g_n\left(r_n \exp{\i \varphi}\right)$ is bounded there by a constant uniform in $n$, 
and then by bounding $g_n\left(r_n \exp{\i \varphi}\right)$ through its modulus we have
        \begin{align}
            \Re{g_n(r_n e^{\i\varphi})} \leq \Re{g(r_n e^{\pm \i\delta_n})}\left(1+\O{(n^*)^{-\alpha/2}}\right).
        \end{align} 
\end{enumerate}
\end{description}
\end{proof}
 \begin{remark}
 \label{rem:general_weights}
One can use a similar argumentation as in the proof of Lemma~\ref{lemma:funct_admiss} for the more general weights  $\vth_m\sim (\log m)^j m^\alpha$.
For this has to replace $p_n$ to be the solution of 
\begin{align}
 n = \sum_{k=1}^\infty \theta_k e^{-kp_n}.
\end{align}
In other words, $e^{-p_n}$ is the saddle point solution for the function $g_\theta(t)$. 
Furthermore it is strait forward to see that the computations in \eqref{eq:saddle_comp_1} and \eqref{eq:saddle_comp_2} are similar and that one has to choose
$q_n$ to be the solution of the equation
 \begin{align*}
 &q_n \left(p_n\right)^{-\alpha-1}\left( \left(\frac{\partial}{\partial \alpha}-\log p_n\right)^j+ j\left(\frac{\partial}{\partial \alpha}-\log \left(p_n\right)\right)^{j-1}  \right)\Gamma(\alpha+1)\\
 =&
  -s^\ast\sum_{k=\floor{x^*}}^{+\infty}\theta_k\exp{-k p_n(1-sq_n)}.
 \end{align*}
\end{remark}
In order to show Thms.~\ref{thm:limit_shape_saddle}, \ref{thm:cumulants} and \ref{thm:beh_increments_saddle} we need to prove first an auxiliary proposition.
\begin{proposition}\label{prop:approx_scaling}
Let $v_n$, $\overline n$, $n^*$, $x^*$ and $s^*$ be as in Lemma~\ref{lemma:funct_admiss}. 
We then have as $n\to +\infty$
\begin{align}\label{eq:expansion_sum}
&
(\exp{-s^*}-1)\sum_{k\geq\lfloor x^* \rfloor} \frac{(\log k)^j k^{\alpha-1} }{\Gamma(\alpha+1)}\exp{-kv_n}\\
& = \left(-s (\overline n)^{\frac{1}{2}}\frac{\Gamma(\alpha,x )}{\Gamma(\alpha+1)}(1+\o{1})
  +\frac{s^2}{2}\frac{\Gamma(\alpha,x )}{\Gamma(\alpha+1)}-\frac{\Gamma(\alpha+1,x  )^2}{\Gamma(\alpha+1)\Gamma(\alpha+2)}s^2 \right)+\o{1}.\nonumber
\end{align}
Furthermore, we have
\begin{align}
\label{eq:s^m_second_term}
&[s^m]\left[
(\exp{-s^*}-1)\sum_{k\geq\lfloor x^* \rfloor} \frac{(\log k)^j k^{\alpha-1} }{\Gamma(\alpha+1)}\exp{-kv_n}\right]\\
\sim&\,
(\overline n)^{1-\frac{m}{2}}[s^m]\left[
(\exp{-s}-1)\sum_{k=0}^\infty  \frac{s^k}{k!}\frac{\Gamma(\alpha+k,x)}{\Gamma(\alpha+1)} \left(-\frac{\Gamma(\alpha+1,\,x)}{\Gamma(\alpha+2)}\right)^k
\right].\nonumber
\end{align}

\end{proposition}

\begin{proof}
 We apply Lemma~\ref{lemma:approxsums} with $z_n = x^*=xn^*$, $v_n = \left(\frac{\overline n}{n}\right)(1-sq_n)$, $\delta=\alpha-1$ and $\widetilde{q}_n=sq_n$.
 We then have for any $\ell>m$
 \begin{align}
\sum_{k=\lfloor x^* \rfloor}^\infty \frac{(\log k)^j k^{\alpha-1} }{\Gamma(\alpha+1)}\exp{-k v_n}
=&
\frac{\left(\frac{\overline n}{n}\right)^{-\alpha}}{\Gamma(\alpha+1)}  \left(\sum_{k=0}^{\ell}  \left(\frac{\partial}{\partial \delta} -\log \left(\frac{\overline n}{n}\right) \right)^j \frac{\Gamma(\alpha+k,x)}{k!} \left(-sq_n\right)^{k}  \right)\nonumber\\
\label{eq:approx_scaling_1}
&+ \O{\left(\log \left(\frac{\overline n}{n}\right)\right)^j (sq_n)^{m+1} + \left(\left(\frac{\overline n}{n}\right)\right)^{\delta} \left(\log \left(\frac{\overline n}{n}\right)\right)^j }.
 \end{align}
We now use \eqref{eq:size_pn} and \eqref{eq:r_n_adm}
\begin{align*}
 &[s^m]\left[\frac{\left(\frac{\overline n}{n}\right)^{-\alpha}}{\Gamma(\alpha+1)}  \left(\sum_{k=0}^{m}  \left(\frac{\partial}{\partial \delta} -\log \left(\frac{\overline n}{n}\right) \right)^j \frac{\Gamma(\alpha+k,x)}{k!} \left(-sq_n\right)^{k}  \right)\right]\\
 =& \,
 (-q_n)^m \frac{\left(\frac{\overline n}{n}\right)^{-\alpha}}{\Gamma(\alpha+1)}  \left(\frac{\partial}{\partial \delta} -\log \left(\frac{\overline n}{n}\right) \right)^j \frac{\Gamma(\alpha+m,x)}{m!} \\
 \sim&\,
 (-q_n)^m \frac{\left(\frac{\overline n}{n}\right)^{-\alpha}}{\Gamma(\alpha+1)}  \left(-\log \left(\frac{\overline n}{n}\right)\right)^j \frac{\Gamma(\alpha+m,x)}{m!}
  \sim\,
 (-q_n)^m \frac{\overline n}{\Gamma(\alpha+1)} \frac{\Gamma(\alpha+m,x)}{m!}\\
 \sim&\,
 \frac{(\overline n)^{1-\frac{m}{2}}}{m!}\frac{\Gamma(\alpha+m,x)}{\Gamma(\alpha+1)} \left(-\frac{\Gamma(\alpha+1,\,x)}{\Gamma(\alpha+2)}\right)^m
\end{align*}
It in now easy to see the $O(.)$ in \eqref{eq:approx_scaling_1} is uniform in $s$ for $s$ bounded.
Applying Cauchy's integral formula thus gives
\begin{align}
\label{eq:insert_scary_name_1}
 [s^m]\left[\sum_{k=\lfloor x^* \rfloor}^\infty \frac{(\log k)^j k^{\alpha-1}}{\Gamma(\alpha+1)}\exp{-k v_n}\right]
 \sim 
  \frac{(\overline n)^{1-\frac{m}{2}}}{m!}\frac{\Gamma(\alpha+m,x)}{\Gamma(\alpha+1)} \left(-\frac{\Gamma(\alpha+1,\,x)}{\Gamma(\alpha+2)}\right)^m.
\end{align}
One the other hand we have $[s^m]\left[\exp{s^*}-1\right] = (\overline n)^{\frac{m}{2}}[s^m]\left[\exp{s}-1\right] $. Combining this observation with \eqref{eq:insert_scary_name_1} then proves \eqref{eq:s^m_second_term}.
Equation~\eqref{eq:expansion_sum} then follows from \eqref{eq:s^m_second_term} with a direct computation.
\end{proof}

\begin{proof}[Proof of Proposition~\ref{prop:relies2} and Theorem~\ref{thm:cumulants}]
To determine the behaviour of $G_n$ we would like to use Lemma~\ref{lem:generating_w_n}. By \eqref{eq:generating_w_n}
\begin{align*}
\En{\mathrm{exp}\bigl(-{s^*}{w}_n({x^*}) \bigr)}
= 
\frac{1}{h_n}
[t^n] \left[\mathrm{exp}\left( g_\Theta(t) +(e^{-s^*}-1)\sum_{k=\floor{x^*} }^{+\infty} \frac{\vartheta_k}{k}t^k \right) \right].
\end{align*}
We have shown that $g_n(t)=g_\Theta(t) +(e^{-s^*}-1)\sum_{k=\floor{x^*} }^{+\infty} \frac{\vartheta_k}{k}t^k$ is log-$n$-admissible. 
Therefore Thm.~\ref{thm:generalasymptotic} 
tells us how $G_n$ behaves, and we have more precisely to 
recover three terms. In first place we collect the terms
for the asymptotic of $\exp{g_n(r_n)}$: We use Lemma~\ref{lemma:polylog_asymp} and get 

 \begin{eqnarray}
g_\Theta(r_n)&=& \frac{v_n^{-\alpha}}{\Gamma(\alpha+1)}\left(\frac{\partial}{\partial \alpha} - \log \left(\frac{\overline n}{n}\right) \right)^j\Gamma(\alpha)
+O(1).\label{eq:g_Theta_ex}
\end{eqnarray}
Computing the coefficient of $s^m$ gives
\begin{eqnarray*}
[s^m]\left[g_\Theta(r_n)\right]
&\sim& \frac{\left(\frac{\overline n}{n}\right)^{-\alpha}}{\Gamma(\alpha+1)}\left(- \log \left(\frac{\overline n}{n}\right) \right)^j\Gamma(\alpha)[s^m]\left[(1-sq_n)^{-\alpha}\right]
\\
&=&
\frac{\overline n}{\alpha} [s^m]\left[
\left(1-\frac{s}{(\overline n)^{1/2}}\frac{\Gamma(\alpha+1,\,x)}{\Gamma(\alpha+2)}\right)^{-\alpha}\right]\\
&=&
\frac{(\overline n)^{1-\frac{m}{2}}}{\alpha} [s^m]\left[
\left(1-s\frac{\Gamma(\alpha+1,\,x)}{\Gamma(\alpha+2)}\right)^{-\alpha}\right].
\end{eqnarray*}
We thus obtain
\begin{eqnarray}\label{eq:g_equation}
g_\Theta(r_n)
&=&
\frac{(\overline n)}{\alpha}\bigl(1+o(1)\bigr)+s(\overline n)^{1/2}\,\frac{\Gamma(\alpha+1,\,x)}{\Gamma(\alpha+2)}\bigl(1+o(1)\bigr)\nonumber \\
&&+\frac{s^2}{2}\frac{\Gamma(\alpha+1,\,x)^2}{\Gamma(\alpha+2)\Gamma(\alpha+1)}\bigl(1+o(1)\bigr) 
+\O{s^3(\overline n)^{-\frac{1}{2}}},\nonumber
\end{eqnarray}
where $\bigl(1+o(1)\bigr)$ is independent of $s$.
Furthermore, we get with Proposition~\ref{prop:approx_scaling}
\begin{align*}
(\exp{-s^*}-1)\sum_{k\geq\lfloor x^* \rfloor} \frac{\vth_m }{m}\exp{-kv_n}
 =& -s (\overline n)^{\frac{1}{2}}\frac{\Gamma(\alpha,x )}{\Gamma(\alpha+1)}\bigl(1+o(1)\bigr) \\
 & +s^2\left(\frac{1}{2}\frac{\Gamma(\alpha,x )}{\Gamma(\alpha+1)}-\frac{\Gamma(\alpha+1,x  )^2}{\Gamma(\alpha+1)\Gamma(\alpha+2)}\right)\bigl(1+o(1)\bigr) \\ 
 &+\O{s^3(\overline n)^{-\frac{1}{2}}}.
\end{align*}
Thirdly, we obviously have
$$
-n \log(r_n)=\overline n\bigl(1+o(1)\bigr)-s(\overline n)^{1/2} \frac{\Gamma(\alpha+1,\,x)}{\Gamma(\alpha+2)}\bigl(1+o(1)\bigr).
$$
Finally we can use a similar computation as in Proposition~\ref{prop:approx_scaling} to see that
\begin{eqnarray*}
\log(b(r_n))
&=&
C_1\log(n)+C_2\sum_{k\geq 0}s^k (\overline n)^{-\frac{k}{2}}
\end{eqnarray*}
We thus see that the contribution of $b_n$ to the coefficients of $s^m$ is of lower order.
We combine everything and obtain
\begin{align}
 \En{\mathrm{exp}\bigl(-{s^*}{w}_n({x^*}) \bigr)}
 =
 \mathrm{exp}&\left(-s (\overline n)^{\frac{1}{2}}\frac{\Gamma(\alpha,x )}{\Gamma(\alpha+1)}\bigl(1+o(1)\bigr)\right. \nonumber \\
 &+
 \frac{s^2}{2}\left(\frac{\Gamma(\alpha,x )}{\Gamma(\alpha+1)}-\frac{\Gamma(\alpha+1,x  )^2}{\Gamma(\alpha+1)\Gamma(\alpha+2)}\right)\bigl(1+o(1)\bigr)\nonumber\\
 &+
 \O{s^3(\overline n)^{-\frac{1}{2}}}\label{eq:s_cube}
 \Bigr).
\end{align}
Notice that we do not require the asymptotic behaviour of the coefficient of $s^0$ since 
$\En{\mathrm{exp}\bigl(-{s^*}{w}_n({x^*}) \bigr)}=1$ at $s=0$. 
This completes the proof of Proposition~\ref{prop:relies2} (and Theorem~\ref{thm:limit_shape_saddle}). 
The proof of Theorem~\ref{thm:cumulants} follows the same line and we thus omit it.
\end{proof}

\begin{proof}[Proof of Thm.~\ref{thm:beh_increments_saddle}]
For multiple increments, we can repeat the proof of Theorem~\ref{thm:cumulants} to compute the behaviour of the  
vector $\mathbf w_n(\mathbf x^\ast)$ with $\mathbf w_n(\mathbf x^\ast)$ as in \eqref{eq:increm} with length $\ell\geq2$.
Applying Lemma~\ref{lem:generating_w_n_x_y} to $\mathbf w_n(\mathbf x^\ast)$ shows that we have to consider the function
\begin{align}
 g_\Theta(t) +\sum_{j=1}^\ell (e^{-s_j^\ast}-1)\sum_{k=\floor{x_j^\ast} }^{\floor{x_{j+1^\ast}-1} } \frac{\vartheta_k}{k}t^k.
\end{align}
A computation as in Lemma~\ref{lemma:funct_admiss} shows that this function is also log-$n$-admissible for $r_n =e^{-v_n}$ with
\begin{eqnarray*}
&&v_n:=p_n\left(1-\frac{q_n}{\Gamma(\alpha+2)}\left(s_{\ell} \Gamma(\alpha+1,\,x_{\ell})\right.\right.\\
&&\phantom{aaaaaaaaaaa}+\left.\left. \sum_{k=1}^{\ell-1}s_{\ell-k}\left(\Gamma(\alpha+1,\,x_{\ell-1-k})-\Gamma(\alpha+1,\,x_{\ell-k})\right)\right)\right).
\end{eqnarray*}
for some $p_n$ and $q_n$ with $p_n \sim (n^\ast)^{-1}$ and $q_n \sim (n^\ast)^{-1/2}$.
We deduce from this as in \eqref{eq:g_equation} that
\begin{eqnarray}
&&g_\Theta(r_n)\sim \frac{\overline n}{\alpha}(1+\o{1})-\frac{(\overline n)^{1/2}}{\Gamma(\alpha+2)}\cdot\nonumber\\
&&\left(s_\ell\Gamma(\alpha+1,\,x_\ell)+\sum_{k=1}^{\ell-1}(s_{\ell-k-1}(\Gamma(\alpha+1,\,x_{\ell-k-1})-\Gamma(\alpha+1,\,x_{\ell-k}))\right)(1+\o{1})\nonumber\\
&&+\frac{1}{2\Gamma(\alpha+2)\Gamma(\alpha+1)}\left(s_\ell\Gamma(\alpha+1,\,x_\ell)\right.\nonumber\\
&&+\left.\sum_{k=1}^{\ell-1}(s_{\ell-k-1}(\Gamma(\alpha+1,\,x_{\ell-k-1})-\Gamma(\alpha+1,\,x_{\ell-k}))\right)^2(1+\o{1})\nonumber\\
&&+\o{1}.\label{eq:square}
\end{eqnarray}
Since the coefficients of the form $\left(\exp{-s^*_j}-1\right)\sum_{k=x_j^*}^{x_{j+1}^*-1}\frac{\vth_k}{k}r_n^k$ do not give a contribution to covariances, the mixed terms will stem 
from the expansion of the square in \eqref{eq:square}. In particular we see that the coefficient of $s_i s_j$, for $1\leq j<i<\ell$, is
$$
\frac{\left(\Gamma(\alpha+1,\,x_i)-\Gamma(\alpha+1,\,x_{i+1})\right)\left(\Gamma(\alpha+1,\,x_{j})-\Gamma(\alpha+1,\,x_{j+1})\right)}{2\Gamma(\alpha+1)\Gamma(\alpha+2)}.
$$
\end{proof}
\subsection{Functional CLT for \texorpdfstring{$w_n(\cdot)$}{wn}}
As in the randomized setting, a functional CLT can be obtained here too. Unlike the previous case though we do not have the independence of cycle counts, hence we will have to show the tightness of the fluctuations as in Sec.~\ref{sec:CLT_randomization} in two steps (cf. \cite{Ha90}). The result we aim at is, precisely as before,
\begin{theorem}
\label{thm:func_CLT_saddle}
 The process $\widetilde{w}_n^{\mathbf s}:\R^+\to\R$ \textup{(}see Thm.~ \ref{thm:limit_shape_saddle}\textup{)} converges weakly with respect to $\mathbb{P}_n$ as $n\to\infty$ to a 
continuous process $\widetilde{w}_\infty^{\mathbf s}:\R^+\to\R$ with $\widetilde{w}_\infty^{\mathbf s}(x)\sim \mathcal N(0,\,\left(\sigma_\infty^{\mathbf s}(x)\right)^2)$ and whose increments are \emph{not} independent. The covariance 
structure is given in Thm.~\ref{thm:beh_increments_saddle}.
\end{theorem}

We will proceed as in the proof of Thm.~\ref{thm:func_CLT_pt}. Having shown already the behaviour of the increments in Thm.~\ref{thm:beh_increments_saddle} what we have to tackle now 
is their tightness. The proof's goal is again, analogously as Lemma~\ref{lemma:tightness}. However the evaluation of 
the corresponding expectation on the
LHS of \eqref{eq:goal_2} is more difficult this time; one possible approach is present in \cite{DePi85} and is based on P\'olya's enumeration lemma and the calculation of 
factorial moments of cycle counts. However it is easier to use the argumentation by Hansen in \cite{Ha90} to obtain an expression for the corresponding generating function. We get

\begin{lemma}
\label{lemma:tightness_saddle_generating}
For $0\leq x_1<x\leq x_2$ arbitrary and $x^*:={x n^*}$, $x_1^*:={x_1 n^*}$ and $x_2^*:={x_2 n^*}$
\begin{align}\label{eq:lem_tightness_saddle_generating}
&(\overline n)^2\cdot h_n\En{(\widetilde w_n^{\mathbf s}(x^*)-\widetilde w_n^{\mathbf s}(x_1^*))^2 (\widetilde w_n^{\mathbf s}(x_2^*)-\widetilde w_n^{\mathbf s}(x^*))^2}\\
&=
[t^n]\left[\left((g_{x_1^*}^{x^*}(t)-E_{x_1}^{x})^2 + g_{x_1^*}^{x^*}(t) \right) \left((g_{x^*}^{x_2^*}(t)-E_{x}^{x_2})^2 + g_{x^*}^{x_2^*}(t) \right) \mathrm{exp}(g_\Theta(t)) \right]\nonumber
\end{align}
with
$g_{a}^b(z):=\sum_{a\leq j <b} \frac{\vartheta_j}{j}z^j $ 
and 
 $E_a^b = \En{\widetilde w_n^{\mathbf s}(bn^*)-\widetilde w_n^{\mathbf s}(an^*)}$ for $a<b$.
\end{lemma}

\begin{proof}
We define for $0<t<1$ the measure $\mathbb P_t$ on $\mathfrak S:=\cup_n \mathfrak S_n$ as in Section \ref{sec:randomization}. By repeating the proof of \cite[Lemma 2.1]{Ha90} we see that
$$
\Pbt{\sum k C_k=n}=t^n h_n \exp{g_\Theta(t)}.
$$
Let now $\Psi:\,\mathfrak S \to \C$ with $\Et{|\Psi|}<\infty$ and $\Psi$ only depending one the cycles counts, i.e. $\Psi=\Psi(C_1,C_2,\dots)$.
Mimicking Hansen's strategy, one can prove for such a $\Psi$ that
\begin{equation}\label{eq:equality_Psi} 
\Et{\Psi}\exp{g_\Theta(t)}=\sum_{n\geq 1}t^n h_n \En{\Psi_n}+\Psi(0,\,0,\,0,\ldots),
\end{equation}
where $\Psi_n:\,\mathfrak S_n\to \C$ is defined by $\Psi_n=\Psi(C_1,\,\ldots,\,C_n,\,0,\,0,\,\ldots)$.
The proof of \eqref{eq:equality_Psi} is almost the same as in \cite{Ha90} and we thus omit it.
We now use 
$$
\Psi(C_1,\,C_2,\,\ldots):=
\bigl(\widetilde w_n^{\mathbf s}(x^\ast)-\widetilde w_n^{\mathbf s}(x_1^\ast)\bigr)^2 
\bigl(\widetilde w_n^{\mathbf s}(x_2^\ast)-\widetilde w_n^{\mathbf s}(x^\ast)\bigr)^2.
$$
We use the definition of $w_n^{\mathbf s}$ in Theorem~\ref{thm:limit_shape_saddle} and that 
the $C_k$ are independent with respect to $\mathbb{P}_t$ and get
\begin{align}
&(\overline n)\Et{\bigl(\widetilde w_n^{\mathbf s}(x^\ast)-\widetilde w_n^{\mathbf s}(x_1^\ast)\bigr)^2  }
=
 \Et{\left(\sum\limits_{k=x_1^\ast}^{x_1^\ast}C_k -E_{x_1}^{x}\right)^2}\nonumber\\
=&\,
\Et{\sum\limits_{k=x_1^\ast}^{x^\ast}C_k^2 +
\sum\limits_{\substack{x_1^\ast\leq k,k'<x^\ast\nonumber\\k\neq k'}}C_kC_{k'} 
- 2E_{x_1}^{x}\sum\limits_{k=x_1^\ast}^{x^\ast}C_k
+\left(E_{x_1}^{x}\right)^2}\nonumber\\
=&\,
\Et{\sum\limits_{k=x_1^\ast}^{x^\ast}C_k(C_k-1) +
\sum\limits_{\substack{x_1^\ast\leq k,k'<x^\ast\\k\neq k'}}C_kC_{k'} 
- 2E_{x_1}^{x}\sum\limits_{k=x_1^\ast}^{x^\ast}C_k
+\left(E_{x_1}^{x}\right)^2 +\sum\limits_{k=x_1^\ast}^{x^\ast}C_k }\nonumber\\
=&
\sum\limits_{k=x_1^\ast}^{x^\ast}\left(\frac{\vth_k}{k}t^k\right)^2 
+
\sum\limits_{\substack{x_1^\ast\leq k,k'<x^\ast\nonumber\\k\neq k'}}\frac{\vth_k}{k}t^k\frac{\vth_{k'}}{k'}t^{k'}
-
2E_{x_1}^{x}\sum\limits_{k=x_1^\ast}^{x^\ast}\frac{\vth_k}{k}t^k
+\left(E_{x_1}^{x}\right)^2 +\sum\limits_{k=x_1^\ast}^{x^\ast}\frac{\vth_k}{k}t^k \nonumber \\
=& 
\left(\sum\limits_{k=x_1^\ast}^{x^\ast}\frac{\vth_k}{k}t^k  - E_{x_1}^{x}\right)^2 
+\sum\limits_{k=x_1^\ast}^{x^\ast}\frac{\vth_k}{k}t^k.\label{eq:adjust}
\end{align}
This completes the proof since the cycle counts in both factors are independent.
\end{proof}

We can now prove the tightness of the process $\widetilde w_n^{\mathbf s}(x^*)$. We have
\begin{lemma}
\label{lemma:tightness_saddle} 
We have for $0\leq x_1<x\leq x_2<K$ with $K$ arbitrary 
\begin{equation}\label{eq:goal_2}
\En{(\widetilde w_n^{\mathbf s}(x^*)-\widetilde w_n^{\mathbf s}(x_1^*))^2 (\widetilde w_n^{\mathbf s}(x_2^*)-\widetilde w_n^{\mathbf s}(x^*))^2}=\O{(x_2-x_1)^2}.
\end{equation}
\end{lemma}
\begin{proof}
We want to apply the saddle-point method to the sequence of functions
$$
g_n(t)
:=
\exp{g_\Theta(t)
+
\log\left((g_{x_1^*}^{x^*}(t)-E_{x_1}^{x})^2 + g_{x_1^*}^{x^*}(t) \right)
+
\left((g_{x^*}^{x_2^*}(t)-E_{x_1}^{x})^2 + g_{x^*}^{x_2^*}(t) \right)
}
$$ 
to extract coefficients. Our first target is to show the log-$n$-admissibility for
 $r_n = e^{-p_n}$ with $p_n$ as in of Lemma~\ref{lemma:funct_admiss}.
Note  that the order of magnitude of $\log\left(g_{x_1^*}^{x^*}(r_n)-E_{x_1}^{x})^2 + g_{x_1^*}^{x^*}(r_n)\right)$ (and similarly for $x_2$) is $\O{\log p_n}$, which is
lower 
than that of $g_\Theta$ as of \eqref{eq:g_Theta_ex}. This tells us that the proof of log-$n$-admissibility goes through without major modifications. Hence we can safely use 
\eqref{eq:lem_tightness_saddle_generating}. It tells us that 
\begin{eqnarray}
&&\En{(\widetilde w_n^{\mathbf s}(x^*)-\widetilde w_n^{\mathbf s}(x_1^*))^2 (\widetilde w_n^{\mathbf s}(x_2^*)-\widetilde w_n^{\mathbf s}(x^*))^2}\nonumber
\\
&&=\frac1{h_n}
\left(\frac1{\overline n}\right)^2[t^n]\left[\left((g_{x_1^*}^{x^*}(t)-E_{x_1}^{x})^2 + g_{x_1^*}^{x^*}(t) \right)\cdot\right.\nonumber\\
&&\left.\cdot \left((g_{x^*}^{x_2^*}(t)-E_{x}^{x_2})^2 + g_{x^*}^{x_2^*}(t) \right) 
\mathrm{e}^{(g_\Theta(t))} \right]\label{eq:today}.
\end{eqnarray}
Differentiating \eqref{eq:generating_w_n_x_y} with respect to $s_1$ and substituting $s_1=0$ shows that
 \begin{align}
   E_{x_1}^x = \En{\widetilde w_n^{\mathbf s}(x^*)-\widetilde w_n^{\mathbf s}(x_1^*)}
   =
   \frac{1}{h_n}[t^n]\left[ g_{x^*}^{x_2^*}(t) \mathtt{exp}(g_\Theta(t))\right]
 \end{align}
This function is again log-$n$-admissibility with the same $r_n$. It is thus straightforward to show that $g_{x_1^*}^{x^*}(r_n)-E_{x_1}^x =\o{x-x_1}$. 
Therefore 
$$
\left((g_{x_1^*}^{x^*}(r_n)-E_{x_1}^{x})^2 + g_{x_1^*}^{x^*}(r_n)\right)(\overline n)^{-1}=\O{g_{x_1^*}^{x^*}(r_n)(\overline n)^{-1}}.
$$
As it was shown in the proof of Lemma \ref{lemma:tightness} we have that $g_{x_1^*}^{x^*}(r_n)(\overline n)^{-1}=\O{(x-x_1)}$. Similar considerations apply for $x_2$. Hence we can say that 
the RHS of \eqref{eq:today} yields
$$
h_n^{-1}[t^n]\left[\mathrm{exp}(g_\Theta(t)) \right]\O{(x_2-x_1)^2}=\O{(x_2-x_1)^2}.
$$
\end{proof}

\section*{Acknowledgments} We thank Sabine Jansen for pointing out a mistake in a previous version of the paper. The research of  D. Z. was partially supported by grant CRC 701 (Bielefeld University).

\appendix

\section{Euler Maclaurin formula with non integer boundaries}
\label{sec:eulermac}

We prove in this section a slight extension of Euler Maclaurin formula, which allows to deal also with non-integer summation limits.

\begin{theorem}\label{thm:Apostol}
 Let $f:\R\to\R$ be a smooth function, $B_k(x)$ be the Bernoulli polynomials and $c<d$ with  $c,d\in\R$. 
We then have for $p\in\N$
\begin{align}
 \sum_{\lfloor c\rfloor \leq k < d} f(k)
 =&
\int_c^d f(x)\,  \mathrm d x - B_1(d-\lfloor d\rfloor) f(d) - B_1(c-\lfloor c\rfloor) f(c)
\\
&+ 
\sum_{k=1}^p (-1)^{k+1} \frac{B_{k+1}(d-\lfloor d\rfloor) f^{(k)}(d) - B_{k+1}(c-\lfloor c\rfloor) f^{(k)}(c)}{k!}\nonumber\\
&+
\frac{(-1)^{p+1}}{(p+1)!} \int_c^d B_{p+1}(x-\lfloor x\rfloor) f^{(p+1)}(x)\, \mathrm d x\nonumber
\end{align}
\end{theorem}

\begin{proof}
The proof of this theorem follows the same lines as the proof of the Euler–Maclaurin summation formula with integer summation limits, see for instance \cite[Theorem 3.1]{Ap84}. 
We give it here though for completeness. Our proof considers only the case $d\notin\Z$. The argumentation for $d\in\Z$ is completely similar.
One possible definition of the Bernoulli polynomials is by induction: 
\begin{align}
B_0(y) &\equiv 1,\\
B_k^\prime(y)&= kB_{k-1}(y)\ \text{ and } \int_0^1 B_k(y)\, \mathrm d y =1 \text{ for } k\geq 1.
\end{align}
In particular, we have $B_1(y) = y-\frac{1}{2}$.
We now have for $m\in\Z$
\begin{align*}
\int_{m}^{m+1} f(y)\, \mathrm d y
&=
\int_{m}^{m+1} B_0(y-m)f(y)\, \mathrm d y\\
&=
[B_1(y-m)f(y)]|_{y=m}^{m+1} - \int_{m}^{m+1} B_1(y-m)f'(y)\, \mathrm d y\\
&=\frac{1}{2} f(m) +\frac{1}{2} f(m+1) -  \int_{m}^{m+1} B_1(y-\lfloor y\rfloor)f'(y)\, \mathrm d y.
\end{align*}
since $B_1(0)=-\frac{1}{2}$ and $B_1(1)=\frac{1}{2}$. We obtain
\begin{align*}
\sum_{k=\lfloor c\rfloor}^{\lfloor d\rfloor} f(k)
 =&
\int_{\lfloor c\rfloor}^{\lfloor d\rfloor} f(x)\, \mathrm d x  + \frac{1}{2} f(\lfloor c\rfloor) 
+  \frac{1}{2} f(\lfloor d\rfloor) 
+ \int_{\lfloor c\rfloor}^{\lfloor d\rfloor} B_1(y-\lfloor y\rfloor)f'(y)\, \mathrm d y.
\end{align*}
Furthermore, we use 
\begin{align*}
\int_{\lfloor d\rfloor}^{d} f(y)\, \mathrm d y
&=
\frac{1}{2} f(\lfloor d\rfloor) +B_1(d-\lfloor d\rfloor) f(d) -  \int_{\lfloor d\rfloor}^{d} B_1(y-\lfloor y\rfloor)f'(y)\, \mathrm d y.
\end{align*}
and get
\begin{align*}
\sum_{k=\lfloor c\rfloor}^{\lfloor d\rfloor} f(k)
 =&
\int_{\lfloor c\rfloor}^{d} f(x)\, \mathrm d x  + \frac{1}{2} f(\lfloor c\rfloor) 
-B_1(d-\lfloor d\rfloor) f(d)
+ \int_{\lfloor c\rfloor}^{d} B_1(y-\lfloor y\rfloor)f'(y)\, \mathrm d y.
\end{align*}
The argumentation for replacing $\lfloor c\rfloor$ by $c$ is similar. One gets 
\begin{align*}
\sum_{\lfloor c\rfloor\leq k < d} f(k)
 =&
\int_{c}^{d} f(x)\, \mathrm d x  -B_1(c-\lfloor c\rfloor) f(c) 
-B_1(d-\lfloor d\rfloor) f(d)\\
&+ \int_{c}^{d} B_1(y-\lfloor y\rfloor)f'(y)\, \mathrm d y.
\end{align*}
The theorem now follows by successive partial integration of $\int_{c}^{d} B_1(y-\lfloor y\rfloor)f'(y)\, \mathrm d y$.
\end{proof}

\bibliographystyle{acm}
\bibliography{literatur}

\end{document}